\documentclass[a4paper,12pt]{amsart}

\usepackage{amsfonts}
\usepackage{amsmath}
\usepackage{amssymb}
\usepackage{mathrsfs}
\usepackage[colorlinks]{hyperref}
\usepackage{graphicx}
\usepackage{algorithm}
\usepackage{listings}
\usepackage{algorithm,algorithmic,enumerate}
\usepackage{amsmath,tikz}
\usetikzlibrary{matrix}

\numberwithin{equation}{section}
\theoremstyle{plain}
\newtheorem{thm}{Theorem}[section]

\newtheorem{cor}[thm]{Corollary}
\newtheorem{lemma}[thm]{Lemma}

\newtheorem{claim}[thm]{Assumption}

\newtheorem{problem}[thm]{Problem}

\theoremstyle{definition}
\newtheorem{defi}[thm]{Definition}
\newtheorem{rem}[thm]{Remark}
\newtheorem{ex}[thm]{Example}

\begin{document}

\title[Inverse source problem]{Heat Source Determining Inverse Problem for non-local in time equation}

\date{\today}

\author[D. Serikbaev]{Daurenbek Serikbaev}
\address{
  Daurenbek Serikbaev:
  \endgraf   
   Al--Farabi Kazakh National University
  \endgraf
  71 Al-Farabi ave., 050040, Almaty, Kazakhstan
  \endgraf
  and
  \endgraf
  Department of Mathematics: Analysis, Logic and Discrete Mathematics
  \endgraf
  Ghent University, Belgium
  \endgraf
  and
  \endgraf   
  Institute of Mathematics and Mathematical Modeling
  \endgraf
  125 Pushkin str., 050010 Almaty, Kazakhstan
  \endgraf 
  {\it E-mail address} {\rm daurenbek.serikbaev@ugent.be}
  }
  
\author[M. Ruzhansky]{Michael Ruzhansky}
\address{
  Michael Ruzhansky:
  \endgraf
Department of Mathematics: Analysis,
Logic and Discrete Mathematics
  \endgraf
Ghent University, Belgium
  \endgraf
 and
  \endgraf
 School of Mathematical Sciences
 \endgraf
Queen Mary University of London
\endgraf
United Kingdom
\endgraf
  {\it E-mail address} {\rm michael.ruzhansky@ugent.be}
 }

\author[N. Tokmagambetov ]{Niyaz Tokmagambetov }
\address{
  Niyaz Tokmagambetov:
  \endgraf 
  Centre de Recerca Matem\'atica
  \endgraf
  Edifici C, Campus Bellaterra, 08193 Bellaterra (Barcelona), Spain
  \endgraf
  and
  \endgraf   
  Institute of Mathematics and Mathematical Modeling
  \endgraf
  125 Pushkin str., 050010 Almaty, Kazakhstan
  \endgraf  
  {\it E-mail address:} {\rm tokmagambetov@crm.cat; tokmagambetov@math.kz}
  }

\date{\today}

\thanks{This research was funded by the Science Committee of the Ministry of Education and Science of the Republic of Kazakhstan (Grant No. AP14872042), by the FWO Odysseus 1 grant G.0H94.18N: Analysis and Partial Differential Equations, and by the Methusalem programme of the Ghent University Special Research Fund (BOF) (Grant number 01M01021). MR is also supported by EPSRC grant EP/R003025/2. NT is also supported by the Beatriu de Pin\'os programme and by AGAUR (Generalitat de Catalunya) grant 2021 SGR 00087.}


\maketitle

\begin{abstract}
In this paper, we consider the inverse problem of determining the time-dependent source term in the general setting of Hilbert spaces and for general additional data. We prove the well-posedness of this inverse problem by reducing the problem to an operator equation for the source function. 
\end{abstract}

\tableofcontents

\section{Introduction}
Let $T > 0$ be an arbitrarily large constant and  $\mathcal{H}$ be a separable Hilbert space. The object of this paper is the following time-fractional heat equation for a general operator $\mathcal{L}$:
\begin{equation}\label{EQ:ISP}
\mathcal{D}_t^\alpha u(t)+a(t)\mathcal{L}u(t)=r(t)g,\; \text{in}\; \mathcal{H},
\end{equation}
for $t\in(0,T].$ Here $\mathcal{D}_t^\alpha$ is the Caputo derivative of the order $\alpha\in(0,1)$ defined by
$$\mathcal{D}_t^\alpha u(t)=\frac{1}{\Gamma(1-\alpha)}\int_0^t (t-\tau)^{-\alpha}\frac{d}{d\tau} u(\tau)d\tau,$$
with the Gamma function $\Gamma(\cdot).$ And, $\mathcal{L}$ is a self-adjoint operator with a discrete spectrum $\{\lambda_\xi>0:\;\xi\in\mathcal{I}\}$ on $\mathcal{H}$, where $\mathcal{I}$ is a countable set, and the system of eigenfunctions $\{\omega_\xi\in\mathcal{H}:\xi\in\mathcal{I}\}$ of the operator $\mathcal{L}$ forms an orthonormal basis in the space $\mathcal{H}.$ 

In this paper, for the equation \eqref{EQ:ISP} we will study the following inverse source problem:
\begin{problem}\label{P:ISP} Given $a(t),\,g$ and an initial function
\begin{equation}\label{CON:IN Source}
    u(0)=h\;\text{in}\; \mathcal{H},
\end{equation}
and an additional data
\begin{equation}\label{CON:ADD Source}
    F[u(t)]=E(t),\; t\in[0,T],
\end{equation}
find a pair of functions $\{r(t),u(t)\}$ satisfying the equation \eqref{EQ:ISP} and conditions \eqref{CON:IN Source}--\eqref{CON:ADD Source}.
\end{problem}

In \eqref{CON:ADD Source}, $F$ is a \textbf{linear bounded functional}
$$F:\mathcal{H}^{1+\gamma}\rightarrow \mathbb{R}.$$ 
Here $\mathcal{H}^{1+\gamma}=\{u\in\mathcal{H}:\,\mathcal{L}^{1+\gamma}u\in\mathcal{H}\}$ and $\gamma\geq 0,$ and $F$ satisfies the following assumption:
\begin{equation}\label{gamma F}
    \left\{\frac{F[\omega_\xi]}{\lambda_\xi^\gamma}\right\}\in l^2(\mathcal{I}).
\end{equation}

As an illustration, we give the following example. 
\begin{ex}
Fix $\mathcal{H}\equiv L^2(0, 1)$ and 
$$
\mathcal{L}u=-u_{xx},\; x\in (0,1),
$$ 
with the homogeneous Dirichlet boundary conditions. The systems of eigenvalues and eigenfunctions are $\{(\pi k)^2\}_{k\in\mathbb{N}}$ and $\{\sqrt{2}\sin{k\pi x}\}_{k\in\mathbb{N}}$, respectively.

Assume that the additional data is given by the functional $F[u(t,\cdot)]:=\int_{0}^1 u(t,x)dx$. Then 
\begin{equation*}F[\omega_k]=\int_0^1 \sqrt{2}\sin{k\pi x}dx=\frac{\sqrt{2}(1+(-1)^{k+1})}{k\pi}=\begin{cases}
    &\frac{2\sqrt{2}}{k\pi},\;\text{if},\;k=2n-1\;(n\in\mathbb{N});\\
    &0, \;\text{if},\;k=2n\;(n\in\mathbb{N}),
\end{cases}\end{equation*} 
and 
$$
\sum_{k\in\mathbb{N}} |F[\omega_\xi]|^2<\infty.
$$ 
From this we see that $\gamma$ in \eqref{gamma F} can be taken to be $0.$  
\end{ex}


Without loss of generality, we can suppose
\begin{equation*}\label{Positiveness F}
   \text{that}\;F\not\equiv 0\; \text{and that}\; F[\omega_\xi]\geq 0, \text{for all}\;\xi\in\mathcal{I}.
\end{equation*}


Otherwise, if $F[\omega_\zeta]<0$ for some $\zeta\in\mathcal{I}$, we can replace $\omega_\zeta$ by $-\omega_\zeta$. Then such $-\omega_\zeta$ satisfy all the properties we need, such as it is still an eigenfunction of $\mathcal{L}$ corresponding to the eigenvalue $\lambda_\zeta$, giving an orthonormal basis of $\mathcal{H},$ and $F[-\omega_\zeta]\geq 0$.

For Problem \ref{P:ISP}, we obtain the following results:
\begin{itemize}
    \item The existence of the solution;
    \item The uniqueness of the solution;
    \item The continuous dependence on the data.
\end{itemize}

The forward Cauchy problem for the model equation \eqref{EQ:ISP} was studied in \cite{SRT23} in the general setting. In this paper, we will use freely the analysis developed in \cite{SRT23}. 
To the best of our knowledge, since then there has been no work that considered inverse source problems such as ours for the model equation \eqref{EQ:ISP} for $\mathcal{L}$ and given additional information by functional $F.$ Therefore, here we will mention only the scientific works closest to our research, which have been done for particular cases of $\mathcal{L}$ and $F$. 

When $a(t)=1$ and $h=0$, in the last section of \cite{SY11} the inverse problem of determining the time-dependent source function was considered for the model case of \eqref{EQ:ISP} with additional information given at some interior fixed point of $\Omega.$ Using this additional information the authors showed only a stability estimate for $r.$ In this paper, we reduce Problem \ref{P:ISP} to the operator equation for $r(t).$ To this end, by using the additional condition \eqref{CON:ADD Source}, we introduce an operator $K$ with the fixed point at $r(t)$. We prove that the unique fixed point of the operator equation is the unknown coefficient $r(t)$ by energy estimate and show its existence by using Shauder's fixed point theorem. 
In addition, we show the stability estimate of $r$ when $h=0.$ 

If $a(t)=constant$ the existence of a solution to the inverse problem of determining the time-dependent source function in time-fractional equations is widely used, the proof is established by reducing the problem to the Volterra integral equation using the additional information references given in \cite{AKM13, IC16, OP22, LL21}. Unlike in our case, this method is not applicable because we cannot obtain explicit equality such as the Volterra integral equation for $r(t)$, we can only work with estimates for it. In \cite{HV22a} the authors considered the inverse problem identifying time-dependent source function in a time-fractional heat equation with a general second-order linear differential operator given by $\mathcal{L}g(x)=\sum_{i,j=1}^{d}\partial_j(a_{ij}(t,x)\partial_j g(x))+c(t)g(x),\;x\in \Omega.$ The time-dependent source function is recovered from an additional integral measurement. The authors established the existence, uniqueness, and regularity of the weak solution. The existence of a weak solution was obtained using Rothe’s method. Meanwhile, well-posedness results in the same problem for multi-term time-fractional diffusion equations are obtained in \cite{HV22b}. 

In our research, the existence of a fixed point to the above-mentioned operator $K$ will be proved using Shauder's fixed point theorem. This is the first attempt to solve the model equation \eqref{EQ:ISP} by this method.

This paper follows the following structure. In Section \ref{Pre}, we collect some preliminary results on fractional calculus, the mean value theorem, and the definitions of spaces. The necessary results of the direct problem for our study are given at the beginning of Section \ref{Main}. The rest of this section deals
with the inverse problem of recovering the pair $(r,u)$. Specifically, an operator $K$ is introduced; then the existence and uniqueness of its fixed point are proved. The stability estimate of $r$ is also given. Moreover, continuous dependence of the solution to Problem \ref{P:ISP} is shown. In Sections \ref{S:L} and \ref{S:F}, we give examples of the operator $\mathcal{L}$ and of the functional $F,$ respectively. 
In Section \ref{S:gamma}, for the particular cases of $\mathcal L$ and ${F}$ we show how to find the value of $\gamma$ in \eqref{gamma F}.
In the last section, we include as an Appendix some classical theorems which are used in this paper.

\section{Preliminary}
\label{Pre}

\subsection{Mittag-Leffler function} In this subsection, we will introduce the Mittag-Leffler function and give its necessary properties which we will use in our investigation. The Mittag-Leffler function is a two-parameter function defined as
$$E_{\alpha,\beta}(z)=\sum_{k=0}^\infty \frac{z^k}{\Gamma(k\alpha+\beta)},\;z\in\mathbb{C}.$$
It generalizes the natural exponential function in the sense that $E_{1,1}=e^z.$

Note that in \cite{Sim14} the following estimate for the Mittag-Leffler function is proved, when $0<\alpha<1$ (not true for $\alpha\geq 1$)
\begin{equation*}
\frac{1}{1+\Gamma(1-\alpha)z}\leq E_{\alpha,1}(-z)\leq\frac{1}{1+\Gamma(1+\alpha)^{-1}z},\;z>0,
\end{equation*}
Thus, it follows that
\begin{equation}\label{EST: Mittag}
0<E_{\alpha,1}(-z)<1,\;z>0.
\end{equation}
\begin{lemma}\label{L:Der E}
For $\lambda>0,\,\alpha>0$ and $n\in\mathbb{N}^{+},$ we have
$$\frac{d^n}{dt^n}E_{\alpha,1}(-\lambda t^\alpha)=-\lambda t^{\alpha-n}E_{\alpha,\alpha-n+1}(-\lambda t^\alpha),\;t>0.$$
In particular, if we set $n=1,$ then we have
$$\frac{d}{dt}E_{\alpha,1}(-\lambda t^\alpha)=-\lambda t^{\alpha-1}E_{\alpha,\alpha}(-\lambda t^\alpha),\;t>0.$$
\end{lemma}
\begin{proof}
This is \cite[Lemma 3.2]{SY11}. 
\end{proof}
\begin{lemma}\label{EST:1/Gamma}\cite{BZhSY16}
For $0<\alpha<1,$ the Mittag-Leffler type function $E_{\alpha,\alpha}(-\lambda t^\alpha)$ satisfies
$$0\leq E_{\alpha,\alpha}(-\lambda t^\alpha)\leq\frac{1}{\Gamma(\alpha)},\; t\in[0,\infty),\; \lambda\geq 0.$$
\end{lemma}

\subsection{Mean value theorem} In this part, we present the generalized mean value theorem. 
\begin{thm}\label{Cor}\cite[Theorem 1]{OSh07} Let $0<\alpha\leq 1,$ $a<b$ and $g\in C[a,b]$ be such that ${}_a\mathcal{D}^\alpha_t g\in C[a,b].$ Then, there exists some $t\in (a,b)$ such that
$$
g(b)-g(a)=\frac{1}{\Gamma(\alpha+1)}{}_a\mathcal{D}^\alpha_t g(t)\cdot (b-a)^\alpha.
$$
\end{thm}
Here ${}_a\mathcal{D}^\alpha_t$ is a Caputo derivative defined by
$${}_a\mathcal{D}^\alpha_t g=\frac{1}{\Gamma(1-\alpha)}\int_a^t (t-\tau)^{-\alpha}\frac{d}{d\tau} g(\tau)d\tau .$$
Theorem \ref{Cor} plays an important role in proving the existence of a solution to our inverse problems.
\subsection{Gronwall type lemma}
We present now an inequality of Gronwall type with weakly singular kernel $(t-\tau)^{\alpha-1}$ (see \cite[Lemma 7.1.1]{H81}).
\begin{lemma}\label{L:Gronwall}
    Suppose $c\geq 0,\,0<\alpha<1,$ and $z(t)$ is a non-negative function locally integrable on $[0,b)$ (some $b\leq\infty$), and suppose that $y(t)$ is non-negative and locally integrable on $[0,b)$ with
    $$y(t)\leq z(t)+c\int_0^t(t-\tau)^{\alpha-1}y(\tau)d\tau,\;\forall t\in[0,b).$$
    Then 
    $$y(t)\leq z(t)+c\Gamma(\alpha)\int_0^t \frac{d}{d\tau}E_{\alpha,1}(c\Gamma(\alpha)(t-\tau)^\alpha)z(\tau)d\tau, \;\forall t\in[0,b).$$
    If $z(t)\equiv z,$ constant, then $$y(t)\leq zE_{\alpha,1}(c\Gamma(\alpha)t^\alpha),\;\forall t\in[0,b).$$
\end{lemma}

\subsection{Sobolev spaces} In this section, we fix a few definitions concerning generalized Sobolev space over $\mathcal{H}.$

\begin{defi}\label{d:H rho}
Let $\rho\in\mathbb{R}$. We denote the Sobolev space by $\mathcal{H}^{\rho}:=\{u\in\mathcal{H}:\; \mathcal{L}^\rho u\in \mathcal{H}\}$ with the norm
$$\|u\|_{\mathcal{H}^\rho}=\left(\sum_{\xi\in\mathcal{I}}|(1+\lambda_\xi)^\rho(u,\omega_\xi)_{\mathcal{H}}|^2\right)^\frac{1}{2}.$$
\end{defi}
\begin{defi}\label{norm}
Let $\rho\in\mathbb{R}$. For $0<\alpha<1$ we denote by ${X}^\alpha([0,T];\mathcal{H}^{\rho})$ the space of all continuous functions $g:[0,T]\rightarrow \mathcal{H}^\rho$ with also continuous $\mathcal{D}^\alpha_t g:[0,T]\rightarrow \mathcal{H}^\rho$, such that
$$
\|g\|_{{X}^\alpha([0,T];\mathcal{H}^{\rho})}:=\|g\|_{C([0,T];\mathcal{H}^{\rho})}+\|\mathcal{D}_t^\alpha g\|_{C([0,T];\mathcal{H}^{\rho})}<\infty.
$$
The space ${X}^\alpha([0,T];\mathcal{H}^{\rho})$ equipped with the norm above is a Banach space.
\end{defi}

\section{Main results}
\label{Main}

During the course of our study, the given data $a,\,h,\,g$ and $E$ are supposed to satisfy the following:
\begin{claim}\label{A:ISP} Let $\gamma$ be a constant satisfying \eqref{gamma F}.
    
$(a)$ $a\in C^+[0,T]:=\{a\in C[0,T]:\;a(t)\geq q_a>0,\,t\in[0,T]\};$

$(b)$ $g\in \mathcal{H}^{1+\gamma}$ with $F[g]\neq 0;$

$(c)$ $h\in \mathcal{H}^{2+\gamma};$

$(d)$ $E\in X^\alpha[0,T]$ and $E(t)\neq 0$ for all $t\in [0,T].$
\end{claim}

To solve Problem \ref{P:ISP} we will use widely the results for the Cauchy problem \eqref{EQ:ISP}--\eqref{CON:IN Source} when $r(t)$ is given and belongs to $C[0,T].$ For our convenience, below we give some results on the direct problems from the paper \cite{SRT23}. 
\begin{lemma}\label{L:1}
    Define $g_\xi=(g,\omega_\xi)_\mathcal{H},\;h_\xi=(h,\omega_\xi)_\mathcal{H}$ for $\xi\in\mathcal{I}.$ Then the representation of the generalized solution to \eqref{EQ:ISP}--\eqref{CON:IN Source} is given by    
    \begin{equation}\label{expantion u source}
u(t;a;r)=\sum_{\xi\in\mathcal{I}}u_\xi(t;a;r)\omega_\xi,
    \end{equation}
    where $u_\xi(t;a;r)$ satisfies the following time-fractional equation 
    \begin{equation}\label{FODE source}
        \begin{cases}
            &\mathcal{D}_t^\alpha u_\xi(t;a;r)+\lambda_\xi a(t) u_\xi(t;a;r)=r(t)g_\xi,\;t\in(0,T],\\
            &u(0;a;r)=h_\xi,
        \end{cases}
    \end{equation}
for each $\xi\in\mathcal{I},$
where the notation $u(t;a;r)$ is used for displaying the dependence of the solution $u$ to \eqref{EQ:ISP}--\eqref{CON:IN Source} on the functions $a(t)$ and $r(t).$
\end{lemma}
\begin{proof}
    See \cite[Lemma 3.3]{SRT23}.
\end{proof}

\begin{lemma}\label{L:2}
    Let $a\in C^{+}[0,T],\;h\in \mathcal{H}^1$ and $g\in\mathcal{H}.$ If $r(t)$ is given and belongs to $C[0,T]$ then the following statements are true:
    
    1) For each $\xi\in\mathcal{I}$ there is a unique solution $u_\xi(t;a;r)$ of \eqref{FODE source} which belongs to $X^\alpha[0,T];$

    2) The equation \eqref{EQ:ISP} has a unique generalized solution $u(t;a;r)$ of the form \eqref{expantion u source}.
\end{lemma}
\begin{proof}
   See \cite[Theorem 3.4]{SRT23}.
\end{proof}

We split the problem \eqref{EQ:ISP}--\eqref{CON:IN Source} into two problems:
\begin{equation}\begin{split}\label{EQ:ur source}
  \mathcal{D}_t^\alpha u(t)+a(t)\mathcal{L}u(t)&=r(t)g,\;\text{in}\;\mathcal{H},\;\text{for}\; t\in(0,T];\\
  u(0)&=0\; \text{in}\;\mathcal{H}, 
\end{split}
\end{equation}
and
\begin{equation}\begin{split}\label{EQ:ui source}
  \mathcal{D}_t^\alpha u(t)+a(t)\mathcal{L}u(t)&=0,\;\text{in}\;\mathcal{H},\;\text{for}\; t\in(0,T];\\
  u(0,x)&=h\;\text{in}\;\mathcal{H}.  
\end{split}
\end{equation}

Denote the generalised solutions of the problems \eqref{EQ:ur source} and \eqref{EQ:ui source} by $u^r(t;a;r)$ and $u^i(t;a),$
respectively ($``r"$ and $``i"$ stand for ``right-hand side" and ``initial condition"). 

In the following statement, we give the main properties of $u^r(t;a;r)$ and $u^i(t;a)$ which directly follow from Lemmas \ref{L:1} and \ref{L:2}.

\begin{lemma}\label{L:16}
Suppose Assumption \ref{A:ISP} holds. Then the problems \eqref{EQ:ur source} and  \eqref{EQ:ui source} have unique solutions with the representations as 
\begin{equation*}\label{expan uri source}
u^r(t;a;r)=\sum_{\xi\in\mathcal{I}}u_\xi^r(t;a;r)\omega_\xi,\;u^i(t;a)=\sum_{\xi\in\mathcal{I}}u_\xi^i(t;a)\omega_\xi,  
\end{equation*}
where $u_\xi^r(t;a;r)$ and $u_\xi^i(t;a)$ satisfy the equations
\begin{equation*}\label{FODE ur source}
    \mathcal{D}_t^\alpha u_\xi^r(t;a;r)+\lambda_\xi a(t)u_\xi^r(t;a;r)=r(t)g_\xi,\;u_\xi^r(0;a;r)=0,\;\xi\in\mathcal{I};
\end{equation*}
\begin{equation*}\label{FODE ui source}
    \mathcal{D}_t^\alpha u_\xi^i(t;a)+\lambda_\xi a(t)u_\xi^i(t;a)=0,\;u_\xi^i(0;a)=h_\xi,\;\xi\in\mathcal{I}.
\end{equation*}
\end{lemma}

Lemma \ref{L:16} ensures that the solution of \eqref{FODE source} can be written as 
$$
u_\xi(t;a;r)=u_\xi^r(t;a;r)+u_\xi^i(t;a),\;\xi\in\mathcal{I}.
$$ 

In further investigation, we use the following Corollaries:
\begin{cor}\label{q_a>a ur source}
 Let Assumption \ref{A:ISP} $(a)$ hold true. Then we have
 $$|u^r_\xi(t;a;r)|\leq |g_\xi|\int_0^t |r(s)|(t-s)^{\alpha-1} E_{\alpha,\alpha}(-\lambda_\xi q_a (t-s)^\alpha)ds\;\text{on}\; \;[0,T],\;\xi\in\mathcal{I}.$$
\end{cor}
\begin{proof}
    See \cite[Corollary 3.11]{SRT23}.
\end{proof}
\begin{cor}\label{q_a>a ui source}
    Let Assumption \ref{A:ISP} $(a)$ hold true. Then we have
    $$|u_\xi^i(t;a)|\leq |h_\xi|E_{\alpha,1}(-\lambda_\xi q_a t^\alpha),\;\text{on}\;\;[0,T],\;\xi\in\mathcal{I}.$$
\end{cor}
\begin{proof}
    See \cite[Corollary 3.16]{SRT23}.
\end{proof}
\begin{cor}\label{L:E1E2 ur source}
For $0<\alpha<1,$ the Mittag-Leffler type function $E_{\alpha,\alpha}(-\lambda_\xi q_at^\alpha)$ satisfies
$$0\leq E_{\alpha,\alpha}(-\lambda_\xi q_a t^\alpha)\leq E_{\alpha,\alpha}(-\inf_{\xi\in\mathcal{I}}\lambda_\xi q_a t^\alpha),\; t>0.$$
\end{cor}
\begin{proof}
    See \cite[Corollary 3.12]{SRT23}.
\end{proof}

Using the additional information \eqref{CON:ADD Source} we reduce Problem \ref{P:ISP} to some operator equation for $r(t).$  To this end, we introduce an operator $K.$

\subsection{Operator $K$} The operator $K$ is defined as
\begin{equation}\label{Operator K source}
    K[r](t)=\frac{\mathcal{D}_t^\alpha E+a(t)\sum_{\xi\in\mathcal{I}}\lambda_\xi u_\xi(t;a;r)F[\omega_\xi]}{F[g]},
\end{equation}
with the domain
\begin{equation*}\begin{split}
D:=\biggl\{r\in C[0,T]:\;
\|r\|_{C[0,T]}&\leq C_2\biggr\},
\end{split}\end{equation*}
where
\begin{align*}
C_2=&\biggl(\frac{C_F Q_a\|h\|_{\mathcal{H}^{1+\gamma}}}{|F[g]|}
    +\frac{C_F Q_a\|\mathcal{D}_\tau^\alpha E\|_{C[0,T]}\|g\|_{\mathcal{H}^{\gamma}}}{q_a|F[g]|^2}\biggr)\\
    &\times E_{\alpha,1}\left(\frac{C_F Q_a\|g\|_{\mathcal{H}^{1+\gamma}}}{|F[g]|} T^\alpha\right)+\frac{\|\mathcal{D}_t^\alpha E\|_{C[0,T]}}{|F[g]|},
\end{align*}
where $Q_a$ is a constant such that
\begin{equation}\label{q<Q source}
   0<q_a\leq a(t)<Q_a\;\text{on}\;[0,T].
\end{equation}
Assumption \ref{A:ISP} $(a)$ guaranties the existence of $Q_a>q_a.$
The next remark shows that the equality in the definition of $K$ is valid.

\begin{rem}\label{r:1}
    Let $r\in D$. Then using Corollary \ref{q_a>a ur source}, for each $t\in[0,T]$ we have 
    \begin{equation*}\begin{split}
    &\|u^r(t;a;r)\|^2_{\mathcal{H}^{2+\gamma}}=\sum_{\xi\in\mathcal{I}}|(1+\lambda_\xi)^{2+\gamma}u^{r}_\xi(t;a;r)|^2\\
    &\leq \biggr(\frac{1}{\inf_{\xi\in\mathcal{I}}\lambda_\xi}+1\biggl)^{2(2+\gamma)}\sum_{\xi\in\mathcal{I}}\left|\lambda_\xi^{1+\gamma}|g_\xi|\int_0^t |r(\tau)|\lambda_\xi (t-\tau)^{\alpha-1}E_{\alpha,\alpha}(-\lambda_\xi q_a(t-\tau)^\alpha)d\tau\right|^2\\
    &\leq \biggr(\frac{1}{\inf_{\xi\in\mathcal{I}}\lambda_\xi}+1\biggl)^{2(2+\gamma)}\sum_{\xi\in\mathcal{I}}\left|\lambda_\xi^{1+\gamma}|g_\xi| \|r\|_{C[0,T]}\int_0^t \lambda_\xi(t-\tau)^{\alpha-1}E_{\alpha,\alpha}(-\lambda_\xi q_a(t-\tau)^\alpha)d\tau\right|^2\\
    &\leq \biggr(\frac{1}{\inf_{\xi\in\mathcal{I}}\lambda_\xi}+1\biggl)^{2(2+\gamma)}q_a^{-2}C_2^2\|g\|^2_{\mathcal{H}^{1+\gamma}},
    \end{split}\end{equation*}
    where we use the following estimate
\begin{equation}\label{EST:1/q}
\int_0^t \lambda_\xi (t-\tau)^{\alpha-1}E_{\alpha,\alpha}(-\lambda_\xi q_a (t-\tau)^\alpha)d\tau=\frac{1}{q_a}(1-E_{\alpha,1}(-\lambda_\xi q_a t^\alpha))<\frac{1}{q_a}.
\end{equation}
This estimate is obtained by using Lemma \ref{L:Der E} and \eqref{EST: Mittag}.

Corollary \ref{q_a>a ui source} together with \eqref{EST: Mittag} gives
    \begin{equation*}\begin{split}
    &\|u^i(t;a)\|^2_{\mathcal{H}^{2+\gamma}}=\sum_{\xi\in\mathcal{I}}|(1+\lambda_\xi)^{2+\gamma}u^{i}_\xi(t;a)|^2\\
    &\leq \biggr(\frac{1}{\inf_{\xi\in\mathcal{I}}\lambda_\xi}+1\biggl)^{2(2+\gamma)} \sum_{\xi\in\mathcal{I}}|\lambda_\xi^{2+\gamma}h_\xi E_{\alpha,1}(-\lambda_\xi q_a t^\alpha)|^2\\
    &\leq\biggr(\frac{1}{\inf_{\xi\in\mathcal{I}}\lambda_\xi}+1\biggl)^{2(2+\gamma)}\|h\|^2_{\mathcal{H}^{2+\gamma}}.
    \end{split}\end{equation*}
\end{rem}
These two estimates yield the following result:
\begin{equation}\label{regularity u}
\|u\|_{C([0,T];\mathcal{H}^{2+\gamma})}\leq \biggr(\frac{1}{\inf_{\xi\in\mathcal{I}}\lambda_\xi}+1\biggl)^{2+\gamma}(\|h\|_{\mathcal{H}^{2+\gamma}}+q_a^{-1}C_2\|g\|_{\mathcal{H}^{1+\gamma}}),
\end{equation}
implying that $\|u\|_{\mathcal{H}^{2+\gamma}}<\infty,$ for any $t\in[0,T].$ The series 
\begin{equation*}\begin{split}
\sum_{\xi\in\mathcal{I}}\lambda_\xi u_\xi(t;a;r)F[\omega_\xi]&\leq \left(\sum_{\xi\in\mathcal{I}}\left|\frac{F[\omega_\xi]}{\lambda_\xi^{\gamma}}\right|^2\right)^\frac{1}{2}\left(\sum_{\xi\in\mathcal{I}}|\lambda_\xi^{1-\gamma}u_\xi(t;a;r)|^2\right)^\frac{1}{2}\leq C_F\|u\|_{\mathcal{H}^{1+\gamma}}
\end{split}\end{equation*}
for $t\in[0,T],$ where $C_F=\left(\sum_{\xi\in\mathcal{I}}\left|\frac{F[\omega_\xi]}{\lambda_\xi^{\gamma}}\right|^2\right)^\frac{1}{2},$ is finite by assumption \eqref{gamma F}. 

For the operator $K,$ we have the following lemmas.

\begin{lemma}
The operator $K$ is well-defined.
\end{lemma}
\begin{proof} Let $r\in D.$ Then Lemma \ref{L:2} ensures that there exists a unique $u_\xi(t;a;r)$ for each $\xi\in\mathcal{I},$ which implies with Assumption \ref{A:ISP} and Remark \ref{r:1} the well-definiteness of $K[r](t).$ This completes the proof.
\end{proof}

\begin{lemma}
$K$ maps $D$ to $D$.
\end{lemma}
\begin{proof}
Given $r\in D.$  The continuity of $K [r]$ follows from the continuity of $u_\xi(t;a;r)$ for each $\xi\in\mathcal{I}$ and continuity of $E(t)$  established in Lemma \ref{L:2} and Assumption \ref{A:ISP} $(d)$, respectively.

The functional $F$ is linear and bounded on $\mathcal{H}^{1+\gamma}.$ We assumed the linearity and boundedness of $F$, when we introduced the additional condition \eqref{CON:ADD Source}. 

Acting by the functional $F$ on \eqref{expantion u source} 
we get 
\begin{equation}
\label{FU}F[u(t)]=\sum_{\xi\in\mathcal{I}}u_\xi(t;a)F[\omega_\xi].
\end{equation}
The linearity and boundedness of $F$ on $\mathcal{H}^{1+\gamma}$ in view of Theorem \ref{LFB} give us continuity of $F$ on $\mathcal{H}^{1+\gamma}.$ Here we use the continuity and linearity of $F$ to put the functional $F$ under the sum and to get \eqref{FU}. 

Applying the operator $\mathcal{D}_t^\alpha$ to \eqref{FU}, we have the following
\begin{equation}\label{DFU}
\mathcal{D}_t^\alpha F[u(t)]=\sum_{\xi\in\mathcal{I}}\mathcal{D}_t^\alpha u_\xi(t;a)F[\omega_\xi].
\end{equation}
Assumption \ref{A:ISP} $(d)$ together with \eqref{CON:ADD Source} allow the acting by $\mathcal{D}_t^\alpha$ to $F[u(t)]$ and guarantee its meaningfulness.

Acting by the operator $\mathcal{D}_t^\alpha$ to \eqref{expantion u source} we have
\begin{equation}\label{DU}
\mathcal{D}_t^\alpha u(t;a)=\sum_{\xi\in\mathcal{I}}\mathcal{D}_t^\alpha u_\xi(t;a)\omega_\xi.
\end{equation}

Applying the functional $F$ to \eqref{DU} and taking into account \eqref{DFU} we have
\begin{equation}\label{change order}
    \begin{split}
        F[\mathcal{D}_t^\alpha u(t)]=\mathcal{D}_t^\alpha F[u(t)].
    \end{split}
\end{equation}

Acting by the operator $\mathcal{D}_t^\alpha$ to \eqref{CON:ADD Source} under the Assumption \ref{A:ISP} $(d)$ and taking into account \eqref{change order}, we obtain
\begin{equation*}\label{dif F}
F[\mathcal{D}^\alpha_t u(t)]=\mathcal{D}^\alpha_t E(t),\;t\in[0,T].
\end{equation*}
In combination with the equation \eqref{EQ:ISP}, the last gives us
\begin{equation}\label{EQ:r}\begin{split}
r(t)&=\frac{F[\mathcal{D}^\alpha_t u(t)]+a(t)F[\mathcal{L}u](t)}{F[g]}\\
&=\frac{\mathcal{D}_t^\alpha E}{F[g]}+U(t),
\end{split}    
\end{equation}
where
\begin{equation}\label{EQ:U}
U(t)=\frac{a(t)\sum_{\xi\in\mathcal{I}}\lambda_\xi u_\xi(t;a;r)F[\omega_\xi]}{F[g]}.
\end{equation}
Finally, in order to estimate the function \eqref{EQ:U}, we see that \eqref{q<Q source} with Corollaries \ref{q_a>a ur source}, \ref{q_a>a ui source} yield
\begin{equation*}\label{U}\begin{split}
    |U(t)|&\leq \frac{a(t)}{|F[g]|}\sum_{\xi\in\mathcal{I}}\lambda_\xi |u_\xi(t;a;r)|F[\omega_\xi]\\
    &\leq \frac{a(t)}{|F[g]|}\sum_{\xi\in\mathcal{I}}\lambda_\xi (|u_\xi^i(t;a;r)|+|u_\xi^r(t;a;r)|)F[\omega_\xi]\\
    &\leq \frac{Q_a}{|F[g]|}\sum_{\xi\in\mathcal{I}}\lambda_\xi F[\omega_\xi] \biggl(|h_\xi|E_{\alpha,1}(-\lambda_\xi q_a t^\alpha)\\
    &+|g_\xi|\int_0^t |r(\tau)|(t-\tau)^{\alpha-1}E_{\alpha,\alpha}(-\lambda_\xi q_a (t-\tau)^\alpha)d\tau\biggr).
    \end{split}\end{equation*}
    Using \eqref{EST: Mittag} and \eqref{EQ:r}, one obtains 
    \begin{equation*}\begin{split}
    |U(t)|
    &\leq \frac{Q_a}{|F[g]|}\sum_{\xi\in\mathcal{I}}\lambda_\xi F[\omega_\xi]\biggl(|h_\xi|
    +|g_\xi|\int_0^t \left|\frac{\mathcal{D}_\tau^\alpha E}{F[g]}\right|(t-\tau)^{\alpha-1}E_{\alpha,\alpha}(-\lambda_\xi q_a (t-\tau)^\alpha)d\tau\\
    &+|g_\xi|\int_0^t \left|U(\tau)\right|(t-\tau)^{\alpha-1}E_{\alpha,\alpha}(-\lambda_\xi q_a (t-\tau)^\alpha)d\tau\biggr).
    \end{split}    
\end{equation*}
We estimate each integral term separately. First,
\begin{align*}
    &\int_0^t \left|\frac{\mathcal{D}_\tau^\alpha E}{F[g]}\right|(t-\tau)^{\alpha-1}E_{\alpha,\alpha}(-\lambda_\xi q_a (t-\tau)^\alpha)d\tau\\
    &\leq \frac{1}{\lambda_\xi|F[g]|}\|\mathcal{D}_t^\alpha E\|_{C[0,T]}\int_0^t \lambda_\xi(t-\tau)^{\alpha-1}E_{\alpha,\alpha}(-\lambda_\xi q_a (t-\tau)^\alpha)d\tau\\
    &\leq \frac{1}{\lambda_\xi q_a|F[g]|}\|\mathcal{D}_t^\alpha E\|_{C[0,T]},
\end{align*}
where the last inequality is obtained by \eqref{EST:1/q}.
By Lemma \ref{EST:1/Gamma} we have
\begin{align*}
   &\int_0^t \left|U(\tau)\right|(t-\tau)^{\alpha-1}E_{\alpha,\alpha}(-\lambda_\xi q_a (t-\tau)^\alpha)d\tau\\
   &\leq \frac{1}{\Gamma(\alpha)}\int_0^t \left|U(\tau)\right|(t-\tau)^{\alpha-1}d\tau.
\end{align*}
Combining the above-obtained estimates, one obtains
\begin{equation*}\begin{split}
&|U(t)|
    \leq \frac{Q_a}{|F[g]|}\sum_{\xi\in\mathcal{I}}\lambda_\xi F[\omega_\xi] |h_\xi|+\frac{Q_a\|\mathcal{D}_\tau^\alpha E\|_{C[0,T]}}{q_a|F[g]|^2}\sum_{\xi\in\mathcal{I}}F[\omega_\xi]|g_\xi|\\
    &+\frac{Q_a}{\Gamma(\alpha)|F[g]|}\sum_{\xi\in\mathcal{I}}\lambda_\xi F[\omega_\xi]|g_\xi|\int_0^t |U(\tau)|(t-\tau)^{\alpha-1}d\tau.
    \end{split}\end{equation*}
    Using the Cauchy-Schwartz inequality, we arrive at 
    \begin{equation*}\begin{split}
    &|U(t)|\leq \frac{Q_a}{|F[g]|}\left(\sum_{\xi\in\mathcal{I}}\left|\frac{F[\omega_\xi]}{\lambda_\xi^\gamma}\right|^2\right)^{\frac{1}{2}}\left(\sum_{\xi\in\mathcal{I}}|\lambda_\xi^{1+\gamma}h_\xi|^2\right)^{\frac{1}{2}}\\
    &+\frac{Q_a\|\mathcal{D}_\tau^\alpha E\|_{C[0,T]}}{q_a|F[g]|^2}\left(\sum_{\xi\in\mathcal{I}}\left|\frac{F[\omega_\xi]}{\lambda_\xi^\gamma}\right|^2\right)^{\frac{1}{2}}\left(\sum_{\xi\in\mathcal{I}}|\lambda_\xi^{\gamma}g_\xi|^2\right)^{\frac{1}{2}}\\
    &+\frac{Q_a}{\Gamma(\alpha)|F[g]|}\left(\sum_{\xi\in\mathcal{I}}\left|\frac{F[\omega_\xi]}{\lambda_\xi^\gamma}\right|^2\right)^{\frac{1}{2}}\left(\sum_{\xi\in\mathcal{I}}|\lambda_\xi^{1+\gamma}g_\xi|^2\right)^{\frac{1}{2}}\int_0^t |U(\tau)|(t-\tau)^{\alpha-1}d\tau\\
    &\leq \frac{C_F Q_a\|h\|_{\mathcal{H}^{1+\gamma}}}{|F[g]|}
    +\frac{C_F Q_a\|\mathcal{D}_\tau^\alpha E\|_{C[0,T]}\|g\|_{\mathcal{H}^{\gamma}}}{q_a|F[g]|^2}\\
    &+\frac{C_F Q_a\|g\|_{\mathcal{H}^{1+\gamma}}}{\Gamma(\alpha)|F[g]|}\int_0^t |U(\tau)|(t-\tau)^{\alpha-1}d\tau.
\end{split}    
\end{equation*}
 
Applying  Lemma \ref{L:Gronwall} to $|U(t)|$, one concludes
\begin{equation*}
    |U(t)|\leq E_{\alpha,1}\left(\frac{C_F Q_a\|g\|_{\mathcal{H}^{1+\gamma}}}{|F[g]|} t^\alpha\right)\biggl(\frac{C_F Q_a\|h\|_{\mathcal{H}^{1+\gamma}}}{|F[g]|}
    +\frac{C_F Q_a\|\mathcal{D}_\tau^\alpha E\|_{C[0,T]}\|g\|_{\mathcal{H}^{\gamma}}}{q_a|F[g]|^2}\biggr).
\end{equation*}
Then 
\begin{equation}\label{Norm:U}\begin{split}
    \|U\|_{C[0,T]}&\leq C_1,
\end{split}\end{equation}
where 
\begin{equation*}
  C_1=E_{\alpha,1}\left(\frac{C_F Q_a\|g\|_{\mathcal{H}^{1+\gamma}}}{|F[g]|} T^\alpha\right)\biggl(\frac{C_F Q_a\|h\|_{\mathcal{H}^{1+\gamma}}}{|F[g]|}
    +\frac{C_F Q_a\|\mathcal{D}_\tau^\alpha E\|_{C[0,T]}\|g\|_{\mathcal{H}^{\gamma}}}{q_a|F[g]|^2}\biggr).  
\end{equation*}
In view of \eqref{Norm:U} and using \eqref{Operator K source}, we finalize by
\begin{equation}\label{K<C}
    \begin{split}
    \|K[r]\|_{C[0,T]}&\leq \biggl(\frac{C_F Q_a\|h\|_{\mathcal{H}^{1+\gamma}}}{|F[g]|}
    +\frac{C_F Q_a\|\mathcal{D}_\tau^\alpha E\|_{C[0,T]}\|g\|_{\mathcal{H}^{\gamma}}}{q_a|F[g]|^2}\biggr)\\
    &\times E_{\alpha,1}\left(\frac{C_F Q_a\|g\|_{\mathcal{H}^{1+\gamma}}}{|F[g]|} T^\alpha\right)+\frac{\|\mathcal{D}_t^\alpha E\|_{C[0,T]}}{|F[g]|}=C_2,
    \end{split}
\end{equation}
concluding that $K:D\rightarrow D.$ 
\end{proof}

\subsection{Existence} In order to show the existence of the fixed point of the operator $K$ in $D$, we state the following lemmas.  
\begin{lemma}\label{L:Bounded Source}
The set $K(D)=\{K[r]:\,r\in D\}$ is uniformly bounded. 
\end{lemma}
\begin{proof} The estimate \eqref{K<C} leads that
$$|K[r]|\leq {C_2},$$
for all $r\in D$ and for all $K[r]\in K(D).$ This completes the proof.
\end{proof}
\begin{lemma}\label{L:Equicon Source}
The set $K(D)=\{K[r]:\,r\in D\}$ is equicontinuous.
\end{lemma}
\begin{proof}
In view of \eqref{EQ:U} and using \eqref{Operator K source}, for any $r\in D$ and $\forall t_1,t_2\in [0,T],$ we have 
\begin{equation}\label{Differ K Source}\begin{split}
    \biggl|K[r](t_1)-K[r](t_2)\biggr|&\leq\frac{a(t_1)}{|F[g]|}|N(t_1)-N(t_2)|\\
    &+\frac{|U(t_2)|}{a(t_2)}|a(t_1)-a(t_2)|\\
    &+\frac{1}{|F[g]|}|\mathcal{D}_{t_1}^\alpha E-\mathcal{D}_{t_2}^\alpha E|,
\end{split}\end{equation}
where $N(t)=\sum_{\xi\in\mathcal{I}}\lambda_\xi F[\omega_\xi]u_\xi(t;a;r).$

First, we estimate the difference of $N(t).$ Without loss of generality let us take $t_1<t_2.$ Then by Theorem \ref{Cor} we have 
\begin{equation}\label{ut1-ut2}
    |u_\xi(t_1;a;r)-u_\xi(t_2;a;r)|=\frac{1}{\Gamma(\alpha+1)}|{}_{t_1}\mathcal{D}_t^\alpha u_\xi(t;a;r)|\cdot|t_1-t_2|^\alpha,
\end{equation}
for any $t\in(t_1,\,t_2)$, for all $\xi\in\mathcal{I}$. Here ${}_{t_1}\mathcal{D}_t^\alpha u_\xi(t;a;r)$ is defined in the following way. Let us fix $\xi\in\mathcal{I}.$ Then the definition of the Caputo derivative yields
\begin{equation*}\begin{split}
{}_{t_1}\mathcal{D}_t^\alpha u_\xi(t;a;r)&=\frac{1}{\Gamma(1-\alpha)}\int_{t_1}^t (t-\tau)^{-\alpha}\frac{d}{d\tau} u_\xi(\tau;a;r)d\tau\\
&=\frac{1}{\Gamma(1-\alpha)}\int_{0}^t (t-\tau)^{-\alpha}\frac{d}{d\tau} u_\xi(\tau;a;r)d\tau\\
&-\frac{1}{\Gamma(1-\alpha)}\int_{0}^{t_1} (t-\tau)^{-\alpha}\frac{d}{d\tau} u_\xi(\tau;a;r)d\tau\\
&=\mathcal{D}_t^\alpha u_\xi(t;a;r)-\mathcal{D}_{t_1}^\alpha u_\xi(t_1;a;r).
\end{split}\end{equation*}
We note that in the case $t_1=0$ the following equality holds: $${}_{t_1}\mathcal{D}_t^\alpha u_\xi(t;a;r)=\mathcal{D}_t^\alpha u_\xi(t;a;r).$$
Now in view of \eqref{FODE source} for $t_1>0$ we define ${}_{t_1}\mathcal{D}_t^\alpha u_\xi(t;a;r):$
$$
\mathcal{D}_t^\alpha u_\xi(t;a;r)+\lambda_\xi a(t)u_\xi(t;a;r)=r(t)g_\xi,\; \text{for}\; t\in(t_1,t_2);
$$
$$
\mathcal{D}_{t_1}^\alpha u_\xi(t_1;a;r)+\lambda_\xi a(t_1)u_\xi(t_1;a;r)=r(t_1)g_\xi,\; \text{for}\; t_1\in(0,T].
$$
Subtracting these equations from each other, we get
\begin{equation*}\label{tDu}
{}_{t_1}\mathcal{D}_t^\alpha u_\xi(t;a;r)=r(t)g_\xi-\lambda_\xi a(t)u_\xi(t;a;r)-\left(r(t_1)g_\xi-\lambda_\xi a(t_1)u_\xi(t_1;a;r)\right).
\end{equation*}\
In view of \eqref{q<Q source}, we have
\begin{equation*}\begin{split}
|{}_{t_1}\mathcal{D}_t^\alpha u_\xi(t;a;r)|&\leq |g_\xi| (\max_{t\in[0,T]}|r(t)|+\max_{t_1\in[0,T]}|r(t_1)|)+\lambda_\xi Q_a (|u_\xi(t;a;r)|+|u_\xi(t_1;a;r)|)\\
&\leq 2C_2|g_\xi|+\lambda_\xi Q_a (|u_\xi(t;a;r)|+|u_\xi(t_1;a;r)|).
\end{split}\end{equation*}
This together with \eqref{ut1-ut2} gives us
\begin{equation*}\begin{split}
        &|N(t_1)-N(t_2)|\leq\sum_{\xi\in\mathcal{I}}\lambda_\xi F[\omega_\xi]|u_\xi(t_1;a;r)-u_\xi(t_2;a;r)|\\
        &\leq \frac{|t_1-t_2|^\alpha}{\Gamma(\alpha+1)}\sum_{\xi\in \mathcal{I}}\lambda_\xi F[\omega_\xi]\cdot |{}_{t_1}\mathcal{D}_t^\alpha u_\xi(t;a)|\\
        &\leq \frac{|t_1-t_2|^\alpha}{\Gamma(\alpha+1)}\biggl(2C_2\sum_{\xi\in \mathcal{I}}\lambda_\xi F[\omega_\xi]|g_\xi|\\
        &+Q_a \sum_{\xi\in \mathcal{I}}\lambda_\xi^2 F[\omega_\xi]|u_\xi(t;a;r)|+Q_a\sum_{\xi\in \mathcal{I}}\lambda_\xi^2 F[\omega_\xi]|u_\xi(t_1;a;r)|
        \biggr).
\end{split}\end{equation*} 
By using the Cauchy–Schwarz inequality for the each term, we obtain
\begin{equation}\begin{split}\label{N}
        |N(t_1)-N(t_2)|&\leq \frac{|t_1-t_2|^\alpha}{\Gamma(\alpha+1)}\biggl[2C_2 C_F\left(\sum_{\xi\in \mathcal{I}}|\lambda_\xi^{1+\gamma}g_\xi|^2\right)^\frac{1}{2}\\
        &+Q_a C_F\left(\sum_{\xi\in \mathcal{I}}|\lambda_\xi^{2+\gamma}u_\xi(t;a;r)|^2\right)^\frac{1}{2}+Q_a C_F\left(\sum_{\xi\in \mathcal{I}}|\lambda_\xi^{2+\gamma}u_\xi(t_1;a;r)|^2\right)^\frac{1}{2}\biggr]\\
        &\leq \frac{|t_1-t_2|^\alpha}{\Gamma(\alpha+1)}\biggl(2C_2 C_F\|g\|_{\mathcal{H}^{1+\gamma}}+2Q_a C_F\|u\|_{C([0,T];\mathcal{H}^{2+\gamma})}
        \biggr)\\
        &\leq C_3 |t_1-t_2|^\alpha,
\end{split}\end{equation}
where 
\begin{align*}
    C_3&=\frac{2 C_F}{\Gamma(\alpha+1)}\biggl[C_2\|g\|_{\mathcal{H}^{1+\gamma}}\\
    &+Q_a\biggr(\frac{1}{\inf_{\xi\in\mathcal{I}}\lambda_\xi}+1\biggl)^{2+\gamma}(\|h\|_{\mathcal{H}^{2+\gamma}}+q_a^{-1}C_2\|g\|_{\mathcal{H}^{1+\gamma}})\biggr],
\end{align*}
and the last inequality is obtained by using \eqref{regularity u}.

Fix an arbitrary $\varepsilon>0.$ Since $a(t)$ and $\mathcal{D}_t^\alpha E$ are continuous on $[0,T]$, then for all $t_1,t_2\in[0,T]$ with $|t_1-t_2|<\min\left\{\delta_a(\varepsilon),\,\delta_E(\varepsilon)\right\}$ there exist $\delta_a(\varepsilon)$ and $\delta_E(\varepsilon)$ such that 
\begin{equation}\label{a}
    |a(t_1)-a(t_2)|<\frac{q_a\varepsilon}{3C_1},\;\;|\mathcal{D}_{t_1}^\alpha E-\mathcal{D}_{t_2}^\alpha E|<\frac{|F[g]|}{3}\varepsilon,
\end{equation}
respectively.

Let 
$$
\delta=\min\left\{\delta_a(\varepsilon),\,\delta_E(\varepsilon),\,\left(\frac{|F[g]|}{3C_3Q_a}\varepsilon\right)^{\frac{1}{\alpha}}\right\}.
$$

For $|t_1-t_2|<\delta$ in \eqref{N} one has
\begin{equation}\label{Fin: N}
    |N(t_1)-N(t_2)|<\frac{|F[g]|}{3 Q_a}\varepsilon.
\end{equation}

Substituting \eqref{a} and \eqref{Fin: N} into \eqref{Differ K Source}, we get
$$
|K[r](t_1)-K[r](t_2)|<\varepsilon.
$$
Therefore, the set $K(D)$ is equicontinuous. 
\end{proof}

\begin{thm}\label{existence source}
Suppose that Assumption \ref{A:ISP} is valid. Then there
exists a fixed point of $K$ in $D.$
\end{thm}
\begin{proof}
In view of the Arzela-Ascoli Theorem \ref{Th:AA} and using Lemmas \ref{L:Bounded Source} and \ref{L:Equicon Source} we see that $K(D)$ is relatively compact in $C[0, T].$ Moreover, $K:D\rightarrow D$ and $D$ is a closed convex subset of $C[0,T]$. According to the Shauder fixed point Theorem \ref{Th:ShF}, the equation $$K[r]=r,\;r\in D,$$
has a solution $r=r^*\in D.$
\end{proof}

\subsection{Stability} Here we prove the following stability result for Problem \ref{P:ISP} with the trivial initial condition:
\begin{thm}\label{Thm:ISP}
Let $h=0$ and Assumption \ref{A:ISP} $(a),\,(b),\,(d)$ hold true. Suppose that $u$ satisfies the Cauchy problem \eqref{EQ:ISP}-\eqref{CON:IN Source} for $r\in C[0,T].$ Then there exist constants $C_4,C_5>0$ such that
\begin{equation}\label{EST:4<E<5}
    C_4\|\mathcal{D}^\alpha E\|_{C[0,T]}\leq \|r\|_{C[0,T]}\leq C_5\|\mathcal{D}^\alpha E\|_{C[0,T]}.
\end{equation}
\end{thm}
\begin{proof}
In view of \eqref{EQ:r} and \eqref{EQ:U}, we have
\begin{equation}\label{EQ:DE}\begin{split}
\mathcal{D}_t^\alpha E=r(t)\sum_{\xi\in\mathcal{I}}g_\xi F[\omega_\xi]-a(t)\sum_{\xi\in\mathcal{I}}\lambda_\xi u_\xi(t;a;r)F[\omega_\xi].
\end{split}\end{equation}
Taking into account \eqref{q<Q source}, one obtains
\begin{equation*}\begin{split}
|\mathcal{D}_t^\alpha E|&\leq |r(t)|\sum_{\xi\in\mathcal{I}}F[\omega_\xi]|g_\xi|+Q_a\sum_{\xi\in\mathcal{I}}\lambda_\xi F[\omega_\xi] |u_\xi(t;a;r)|.
\end{split}\end{equation*}
Let us estimate each term separately. By the Cauchy-Schwartz inequality and \eqref{gamma F}, we have the estimate for the first term
\begin{align*}
    |r(t)|\sum_{\xi\in\mathcal{I}}F[\omega_\xi]|g_\xi|\leq C_F\|g\|_{\mathcal{H}^{\gamma}}\|r\|_{C[0,T]}.
\end{align*}
In view of Corollary \ref{q_a>a ur source} for the second term, one obtains
\begin{equation*}\begin{split}
&Q_a\sum_{\xi\in\mathcal{I}}\lambda_\xi F[\omega_\xi] |u_\xi(t;a;r)|\\
&\leq Q_a\sum_{\xi\in\mathcal{I}}\lambda_\xi F[\omega_\xi]|g_\xi|\int_0^t |r(\tau)|(t-\tau)^{\alpha-1}E_{\alpha,\alpha}(-\lambda_\xi q_a (t-\tau)^\alpha) d\tau\\
&\leq Q_a\sum_{\xi\in\mathcal{I}} F[\omega_\xi]|g_\xi|\|r\|_{C[0,T]}\int_0^t\lambda_\xi (t-\tau)^{\alpha-1}E_{\alpha,\alpha}(-\lambda_\xi q_a (t-\tau)^\alpha) d\tau\\
&\leq C_F\frac{Q_a}{q_a}\|g\|_{\mathcal{H}^{\gamma}}\|r\|_{C[0,T]},
\end{split}\end{equation*}
where the last inequality is obtained by using \eqref{EST:1/q} and the Cauchy-Schwartz inequality. Combining the above-obtained estimates, we arrive at
$$|\mathcal{D}_t^\alpha E|\leq C_F\frac{Q_a+q_a}{q_a}\|g\|_{\mathcal{H}^{\gamma}}\|r\|_{C[0,T]},$$
yielding
$$
 C_4\|\mathcal{D}_t^\alpha E\|_{C[0,T]}\leq \|r\|_{C[0,T]},
 $$
where $C_4=\frac{q_a}{(Q_a+q_a)C_F\|g\|_{\mathcal{H}^{\gamma}}}.$
Hence the first inequality in \eqref{EST:4<E<5} is proved.

In view of \eqref{EQ:DE}, we can write $r$ in the following form:
\begin{equation}\label{4444}
    r(t)=\frac{\mathcal{D}_t^\alpha E+a(t)\sum_{\xi\in\mathcal{I}}\lambda_\xi u_\xi(t;a;r)F[\omega_\xi]}{F[g]}.
\end{equation}
Taking into account \eqref{q<Q source} and Corollary \ref{q_a>a ur source}, we have
\begin{align*}
    |r(t)|&\leq \frac{|\mathcal{D}_t^\alpha E|}{|F[g]|}+\frac{Q_a}{|F[g]|}\sum_{\xi\in\mathcal{I}}\lambda_\xi F[\omega_\xi] |u_\xi(t;a;r)|\\
    &\leq \frac{\|\mathcal{D}_t^\alpha E\|_{C[0,T]}}{|F[g]|}+\frac{Q_a}{|F[g]|}\sum_{\xi\in\mathcal{I}}\lambda_\xi F[\omega_\xi]|g_\xi|\int_0^t |r(\tau)|(t-\tau)^{\alpha-1}E_{\alpha,\alpha}(-\lambda_\xi q_a (t-\tau)^\alpha) d\tau.
    \end{align*}
Lemma \ref{EST:1/Gamma} implies
$$
\int_0^t |r(\tau)|(t-\tau)^{\alpha-1}E_{\alpha,\alpha}(-\lambda_\xi q_a (t-\tau)^\alpha) d\tau\leq \frac{1}{\Gamma(\alpha)}\int_0^t |r(\tau)|(t-\tau)^{\alpha-1}d\tau.
$$
Using this and the Cauchy-Schwartz inequality, one obtains
\begin{align*}
    |r(t)|&\leq \frac{\|\mathcal{D}_t^\alpha E\|_{C[0,T]}}{|F[g]|}+\frac{Q_a}{\Gamma(\alpha)|F[g]|}\sum_{\xi\in\mathcal{I}}\lambda_\xi F[\omega_\xi]|g_\xi|\int_0^t |r(\tau)|(t-\tau)^{\alpha-1}d\tau\\
    &\leq \frac{\|\mathcal{D}_t^\alpha E\|_{C[0,T]}}{|F[g]|}+\frac{Q_a}{\Gamma(\alpha)|F[g]|}C_F\left(\sum_{\xi\in\mathcal{I}}|\lambda_\xi^{1+\gamma}g_\xi|^2\right)^\frac{1}{2}\int_0^t |r(\tau)|(t-\tau)^{\alpha-1}d\tau\\
    &\leq \frac{\|\mathcal{D}_t^\alpha E\|_{C[0,T]}}{|F[g]|}+\frac{Q_aC_F\|g\|_{\mathcal{H}^{1+\gamma}}}{\Gamma(\alpha)|F[g]|}\int_0^t |r(\tau)|(t-\tau)^{\alpha-1} d\tau.
    \end{align*}
Applying Lemma \ref{L:Gronwall} to $|r(t)|$, one has
\begin{equation*}
    |r(t)|\leq \frac{\|\mathcal{D}_t^\alpha E\|_{C[0,T]}}{|F[g]|}E_{\alpha,1}\left(\frac{Q_a C_F\|g\|_{\mathcal{H}^{1+\gamma}}}{|F[g]|}t^\alpha\right).
\end{equation*}
This yields
\begin{equation}\label{5555}
\|r\|_{C[0,T]}\leq C_5 \|\mathcal{D}_t^\alpha E\|_{C[0,T]},
\end{equation}
where $C_5=\frac{1}{|F[g]|}E_{\alpha,1}\left(\frac{Q_a C_F\|g\|_{\mathcal{H}^{1+\gamma}}}{|F[g]|}T^\alpha\right).$ That is, the second inequality in \eqref{EST:4<E<5} is proved. Thus, the proof of Theorem \ref{Thm:ISP} is complete.
\end{proof}

\subsection{Uniqueness} In this part, we prove the uniqueness result for the fixed point of $K$.

\begin{thm}\label{uniqueness source}
There is at most one fixed point of $K$ in $D.$
\end{thm}
\begin{proof}
Let $r_1,r_2\in D$ be any fixed points of $K$. Then $u(t;a;r_1)$ and $u(t;a;r_2)$ satisfy the equation \eqref{EQ:ISP} with the Cauchy condition
\begin{equation*}
    \begin{cases}
        &\mathcal{D}_t^\alpha u(t;a;r_1)+a(t)\mathcal{L}u(t;a;r_1)=r_1(t)g,\;u(0;a;r_1)=h;\\
        &\mathcal{D}_t^\alpha u(t;a;r_2)+a(t)\mathcal{L}u(t;a;r_2)=r_2(t)g,\;u(0;a;r_2)=h,
    \end{cases}
\end{equation*}
and the additional information \eqref{CON:ADD Source}:
    $$F[u(t;a;r_1)]=E(t),$$
    $$F[u(t;a;r_2)]=E(t).$$
Applying $\mathcal{D}_t^\alpha$ to these, one gets
$$
F[\mathcal{D}_t^\alpha u(t;a;r_1)]=\mathcal{D}_t^\alpha E,
$$
$$
F[\mathcal{D}_t^\alpha u(t;a;r_2)]=\mathcal{D}_t^\alpha E,
$$
which yields
\begin{equation}\label{2222}
    \begin{cases}
    &\mathcal{D}_t^\alpha w(t;a;r)+a(t)\mathcal{L}w(t;a;r)=r(t)g,\\
    &w(0)=0,\\
    &F[w(t;a;r)]=0, 
    \end{cases}
\end{equation}
and
\begin{equation}\label{1111}
    F[\mathcal{D}_t^\alpha w(t;a;r)]=0.
\end{equation}
where $w(t;a;r)=u(t;a;r_1)-u(t;a;r_2)$ and $r(t)=r_1(t)-r_2(t).$ 
Applying $F$ on both sides of the first equation of \eqref{2222} and taking into account \eqref{1111}, we get
\begin{equation}\label{3333}
    r(t)=\frac{a(t)\sum_{\xi\in\mathcal{I}}\lambda_\xi w_\xi(t;a;r)F[\omega_\xi]}{F[g]},
\end{equation}
where $w_\xi(t;a;r)=u_\xi(t;a;r_1)-u_\xi(t;a;r_2).$
Then we estimate \eqref{3333} by repeating the estimation procedure of \eqref{4444}, that is, from \eqref{4444} till \eqref{5555}. Finally, we get 
$$
\|r\|_{C[0,T]}\leq 0.
$$ 
This implies that $r_1\equiv r_2.$
\end{proof}

\subsection{Continuous dependence of $(r,u)$ on the data} In this section we investigate the dependence of $(r,u)$ on the data.

\begin{thm}\label{continuous source}
Let Assumption \ref{A:ISP} hold true. Then the solution $(r,u)$ of the problem \eqref{EQ:ISP}--\eqref{CON:ADD Source} depends continuously on the data, that is, there exist positive constants $C_{13}$ and $C_{17},$ such that
\begin{equation*}
\begin{split}
&\|\tilde{r}-r\|_{C[0,T]}\\
&\leq C_{13}\left(\|\tilde{a}-a\|_{C[0,T]}+\|\tilde{g}-g\|_{\mathcal{H}^{1+\gamma}}+\|\tilde{h}-h\|_{\mathcal{H}^{2+\gamma}}+\|\tilde{E}-E\|_{X^\alpha[0,T]}\right),
\end{split}
\end{equation*}
and
\begin{align*}
    \|\tilde{u}-u\|_{C([0,T];\mathcal{H})}&\leq C_{17}\big(\|\tilde{a}-a\|_{C[0,T]}\\&+\|\tilde{g}-g\|_{\mathcal{H}^{1+\gamma}}+\|\tilde{h}-h\|_{\mathcal{H}^{2+\gamma}}+\|\tilde{E}-E\|_{X^\alpha[0,T]}\big),
\end{align*}
where $(\tilde{r},\tilde{u})$ is the solution of the inverse problem \eqref{EQ:ISP}--\eqref{CON:ADD Source} corresponding to the set of data $\left\{\tilde{a},\tilde{g},\tilde{h},\tilde{E}\right\},$ that satisfy Assumption \ref{A:ISP}.
\end{thm}

\begin{proof}
Let $\Upsilon=\{a,g,h,E\}$ and $\tilde{\Upsilon}=\left\{\tilde{a},\tilde{g},\tilde{h},\tilde{E}\right\}$ be two sets of data that satisfy Assumption \ref{A:ISP}. And, suppose that $(r,u)$ and $(\tilde{r},\tilde{u})$ are solutions of the inverse problem \eqref{EQ:ISP}--\eqref{CON:ADD Source} corresponding to the data $\Upsilon$ and $\tilde{\Upsilon}$. That is,
\begin{equation}\label{EQ:u inverse source}
  \begin{cases}
  &\mathcal{D}_t^\alpha u(t;a;r)+a(t)\mathcal{L}u(t;a;r)=r(t)g,\\
  &u(0)=h,\\
  &F[u(t;a;r)]=E(t),\\
  \end{cases}
\end{equation}
and
\begin{equation}\label{EQ:ap u inverse source}
\begin{cases}
  &\mathcal{D}_t^\alpha\tilde{u}(t;\tilde{a};\tilde{r})+\tilde{a}(t)\mathcal{L}\tilde{u}(t;\tilde{a};\tilde{r})=\tilde{r}(t)\tilde{g},\\
  &\tilde{u}(0)=\tilde{h},\\
  &F[\tilde{u}(t;\tilde{a};\tilde{r})]=\tilde{E}(t),
  \end{cases}
\end{equation}
respectively. 

Subtracting equations \eqref{EQ:u inverse source} and \eqref{EQ:ap u inverse source} from each other, we have
\begin{equation}\label{EQ:differ u-tilde u source}
\begin{cases}
  &\mathcal{D}_t^\alpha w+\tilde{a}(t)\mathcal{L}w(t)=f(t),\\
  &w(0)=\tilde{h}-h,\\
  &F[w(t)]=\tilde{E}(t)-E(t),
  \end{cases}
\end{equation}
where $w(t)=\tilde{u}(t;\tilde{a};\tilde{r})-{u}(t;a;r)$ and $f(t)=\tilde{r}(t)(\tilde{g}-g)+g(\tilde{r}-r)-(\tilde{a}(t)-a(t))\mathcal{L}u(t;a;r).$ 

Similarly \eqref{expantion u source}, we can write the solution $u(t;\tilde{a};\tilde{r})$ of the problem \eqref{EQ:ap u inverse source} in the form $u(t;\tilde{a};\tilde{r})=\sum_{\xi\in\mathcal{I}}u_\xi(t;\tilde{a};\tilde{r})\omega_\xi.$ This together with \eqref{expantion u source} give us 
\begin{equation}\label{EXPANTION w sourece}
    w(t)=\sum_{\xi\in\mathcal{I}}w_\xi(t;\tilde{a})\omega_\xi,\;t\in [0,T],
\end{equation}
where $w_\xi(t;\tilde{a})=\tilde{u}_\xi(t;\tilde{a};\tilde{r})-{u}_\xi(t;a;r),$ and it satisfies the fractional equation
\begin{equation*}\label{EQ:ode w}
   \mathcal{D}_t^\alpha w_\xi(t;\tilde{a})+\lambda_\xi\tilde{a}(t)w_\xi(t;\tilde{a})=f_\xi(t),\;w_\xi(0;\tilde{a})=\tilde h_\xi-h_\xi,\;\xi\in\mathcal{I}, 
\end{equation*}
where $f_\xi(t)=(f(t),\omega_\xi)_\mathcal{H}.$

Acting by $\mathcal{D}_t^\alpha$ on both sides of the last equation of \eqref{EQ:differ u-tilde u source} and taking into account \eqref{change order}, one obtains
\begin{equation}\label{EQ:FD u-tilde u source}
    F[\mathcal{D}_t^\alpha w(t)]=\mathcal{D}_t^\alpha\tilde{E}-\mathcal{D}_t^\alpha E.
\end{equation}

Applying $F$ on both sides of the first equation of \eqref{EQ:differ u-tilde u source}, we get
\begin{equation*}
  \tilde{r}(t)-r(t)=\frac{F[\mathcal{D}_t^\alpha w(t)]+\tilde{a}(t)F[\mathcal{L}w(t)]+(\tilde{a}(t)-a(t))F[\mathcal{L}u(t;a;r)]-\tilde{r}(t)F[\tilde{g}-g]}{F[g]}. 
\end{equation*}

We can write $g$ and $\tilde{g}$ in the expansion forms, i.e. $g=\sum_{\xi\in\mathcal{I}}g_\xi\omega_\xi$ and $\tilde{g}=\sum_{\xi\in\mathcal{I}}\tilde{g}_\xi\omega_\xi$ respectively. These together with \eqref{expantion u source}, \eqref{EXPANTION w sourece} and \eqref{EQ:FD u-tilde u source} yield that
\begin{equation}\begin{split}\label{r-r}
  \tilde{r}(t)-r(t)&=\frac{1}{F[g]}\biggl(\mathcal{D}_t^\alpha\tilde{E}-\mathcal{D}_t^\alpha E+(\tilde{a}(t)-a(t))\sum_{\xi\in\mathcal{I}}\lambda_\xi F[\omega_\xi]u_\xi(t;{a};r)\\
  &+\tilde{a}(t)\sum_{\xi\in\mathcal{I}}\lambda_\xi F[\omega_\xi]w_\xi(t;\tilde{a})-\tilde{r}(t)\sum_{\xi\in\mathcal{I}}F[\omega_\xi](\tilde{g}_\xi-g_\xi)\biggr). 
\end{split}\end{equation}
To estimate the difference \eqref{r-r}, let us first estimate the series $\sum_{\xi\in\mathcal{I}}\lambda_\xi F[\omega_\xi]w_\xi(t;\tilde{a}).$ In view of Corollaries \ref{q_a>a ur source} and \ref{q_a>a ui source}, we have
\begin{equation*}
    \begin{split}
        &\sum_{\xi\in\mathcal{I}}\lambda_\xi F[\omega_\xi]|w_\xi(t;\tilde{a}|\leq \sum_{\xi\in\mathcal{I}}\lambda_\xi F[\omega_\xi]|\tilde{h}_\xi-h_\xi|E_{\alpha,1}(-\lambda_\xi q_{\tilde{a}}t^\alpha)\\
        &+\sum_{\xi\in\mathcal{I}}\lambda_\xi F[\omega_\xi]|\tilde{g}_\xi-g_\xi|\int_0^t|\tilde{r}(s)|(t-s)^{\alpha-1}E_{\alpha,\alpha}(-\lambda_\xi q_{\tilde{a}}(t-s)^\alpha)ds\\
        &+\sum_{\xi\in\mathcal{I}}\lambda_\xi F[\omega_\xi]\int_0^t|u_\xi(s;a)||\tilde{a}(s)-a(s)|\lambda_\xi(t-s)^{\alpha-1}E_{\alpha,\alpha}(-\lambda_\xi q_{\tilde{a}}(t-s)^\alpha)ds\\
        &+\sum_{\xi\in\mathcal{I}}\lambda_\xi F[\omega_\xi]g_\xi \int_0^t|\tilde{r}(s)-r(s)|(t-s)^{\alpha-1}E_{\alpha,\alpha}(-\lambda_\xi q_{\tilde{a}}(t-s)^\alpha)ds\\
        &=I_1(t)+I_2(t)+I_3(t)+I_4(t).
        \end{split}
\end{equation*}
Now, we estimate each of the four terms separately.

Using \eqref{EST: Mittag} and the Cauchy-Schwartz inequality, one has
\begin{equation*}\begin{split}
I_1(t)&\leq \sum_{\xi\in\mathcal{I}}\lambda_\xi F[\omega_\xi]|\tilde{h}_\xi-h_\xi|\leq C_F\left(\sum_{\xi\in\mathcal{I}}\lambda_\xi^{2(1+\gamma)}|\tilde{h}_\xi-h_\xi|^2\right)^\frac{1}{2}\\
&\leq C_6\|\tilde{h}-h\|_{\mathcal{H}^{1+\gamma}}.
\end{split}\end{equation*}
In view of \eqref{EST:1/q}, we get
\begin{align*}
    I_2(t)&\leq \sum_{\xi\in\mathcal{I}} F[\omega_\xi]|\tilde{g}_\xi-g_\xi|\|\tilde{r}\|_{C[0,T]}\int_0^t\lambda_\xi(t-s)^{\alpha-1}E_{\alpha,\alpha}(-\lambda_\xi q_{\tilde{a}}(t-s)^\alpha)ds\\
    &\leq \frac{1}{q_{\tilde{a}}}\|\tilde{r}\|_{C[0,T]} C_F\left(\sum_{\xi\in\mathcal{I}}\lambda_\xi^{2\gamma}|\tilde{g}_\xi-g_\xi|^2\right)^{\frac{1}{2}}\leq C_7\|\tilde{g}-g\|_{\mathcal{H}^\gamma}. 
\end{align*}

Using Corollary \ref{L:E1E2 ur source}, we have
\begin{align*}
    I_3(t)&\leq \sum_{\xi\in\mathcal{I}}\lambda_\xi F[\omega_\xi]\|\tilde{a}-a\|_{C[0,T]}\int_0^t|u_\xi(s;a)|\lambda_\xi(t-s)^{\alpha-1}E_{\alpha,\alpha}(-\lambda_\xi q_{\tilde{a}}(t-s)^\alpha)ds\\
    &\leq\|\tilde{a}-a\|_{C[0,T]}\int_0^t\left\{\sum_{\xi\in\mathcal{I}}\lambda_\xi^2 F[\omega_\xi]|u_\xi(s;a)|\right\}(t-s)^{\alpha-1}E_{\alpha,\alpha}(-\inf_{\xi\in\mathcal{I}}\lambda_\xi q_{\tilde{a}}(t-s)^\alpha)ds\\
    &\leq \|\tilde{a}-a\|_{C[0,T]}\int_0^tC_F\left(\sum_{\xi\in\mathcal{I}}|\lambda_\xi^{2+\gamma}u_\xi(s;a)|^2\right)^\frac{1}{2}(t-s)^{\alpha-1}E_{\alpha,\alpha}(-\inf_{\xi\in\mathcal{I}}\lambda_\xi q_{\tilde{a}}(t-s)^\alpha)ds\\
    &\leq C_8\|u\|_{C([0,T];\mathcal{H}^{2+\gamma})}\|\tilde{a}-a\|_{C[0,T]}\int_0^t(t-s)^{\alpha-1}E_{\alpha,\alpha}(-\inf_{\xi\in\mathcal{I}}\lambda_\xi q_{\tilde{a}}(t-s)^\alpha)ds\\
    &\leq C_9\|\tilde{a}-a\|_{C[0,T]},
\end{align*}
where the last inequality is obtained by taking into account \eqref{EST:1/q} and \eqref{regularity u}. Lemma \ref{EST:1/Gamma} gives
\begin{align*}
  I_4(t)&\leq \frac{1}{\Gamma(\alpha)}\sum_{\xi\in\mathcal{I}}\lambda_\xi F[\omega_\xi]g_\xi \int_0^t|\tilde{r}(s)-r(s)|(t-s)^{\alpha-1}ds\\
  &\leq \frac{1}{\Gamma(\alpha)}C_F\left(\sum_{\xi\in\mathcal{I}}|\lambda_\xi^{1+\gamma}g_\xi|^2\right)^\frac{1}{2}\int_0^t|\tilde{r}(s)-r(s)|(t-s)^{\alpha-1}ds\\
  &\leq C_{10}\|g\|_{\mathcal{H}^{1+\gamma}}\frac{1}{\Gamma(\alpha)}\int_0^t|\tilde{r}(s)-r(s)|(t-s)^{\alpha-1}ds\\
  &\leq C_{11}\frac{1}{\Gamma(\alpha)}\int_0^t|\tilde{r}(s)-r(s)|(t-s)^{\alpha-1}ds.
\end{align*}

The estimates above obtained for $I_1(t),\,I_2(t),\,I_3(t)$ and $I_4(t)$ yield the following result:
\begin{equation*}
    \begin{split}
        &\sum_{\xi\in\mathcal{I}}\lambda_\xi F[\omega_\xi]|w_\xi(t;\tilde{a})|\\
    &\leq C_6\|\tilde{h}-h\|_{\mathcal{H}^{1+\gamma}}+C_7\|\tilde{g}-g\|_{\mathcal{H}^{\gamma}}+C_9\|\tilde{a}-a\|_{C[0,T]}\\
    &+C_{11}\frac{1}{\Gamma(\alpha)}\int_0^t |\tilde{r}(s)-r(s)|(t-s)^{\alpha-1}ds.
    \end{split}
\end{equation*}
Taking this into account and using H\"older inequality, from \eqref{r-r} we have
\begin{align*}
    |\tilde{r}(t)-r(t)|&\leq \frac{1}{|F[g]|}\biggl[\|\mathcal{D}_t^\alpha\tilde{E}-\mathcal{D}_t^\alpha {E}\|_{C[0,T]}+\|\tilde{a}-a\|_{C[0,T]}C_F\left(\sum_{\xi\in\mathcal{I}}|\lambda_\xi^{1+\gamma}u_\xi(t;a;r)|^2\right)^\frac{1}{2}\\
    &+\|\tilde{a}\|_{C[0,T]}\sum_{\xi\in\mathcal{I}}\lambda_\xi F[\omega_\xi]|w_\xi(t;\tilde{a})|+\|\tilde{r}\|_{C[0,T]}C_F\left(\sum_{\xi\in\mathcal{I}}|\lambda_\xi^{2\gamma}|\tilde{g}_\xi-g_\xi|^2\right)^\frac{1}{2}\biggr]\\
    &\leq \frac{1}{|F[g]|}\biggl[\|\mathcal{D}_t^\alpha\tilde{E}-\mathcal{D}_t^\alpha {E}\|_{C[0,T]}+\|\tilde{a}-a\|_{C[0,T]}C_F\|u(t;a;r)\|_{\mathcal{H}^{1+\gamma}}\\
    &+\|\tilde{a}\|_{C[0,T]}\sum_{\xi\in\mathcal{I}}\lambda_\xi F[\omega_\xi]|w_\xi(t;\tilde{a})|+\|\tilde{r}\|_{C[0,T]}C_F\|\tilde{g}_\xi-g_\xi\|_{\mathcal{H}^\gamma}\biggr]\\
&\leq C_{12}\biggl(\|\tilde{a}-a\|_{C[0,T]}+\|\tilde{g}-g\|_{\mathcal{H}^{1+\gamma}}+\|\tilde{h}-h\|_{\mathcal{H}^{2+\gamma}}\\
    &+\|\tilde{E}-E\|_{X^\alpha[0,T]}+\frac{1}{\Gamma(\alpha)}\int_0^t |\tilde{r}(s)-r(s)|(t-s)^{\alpha-1}ds\biggr).
\end{align*}
Applying Lemma \ref{L:Gronwall} to $|\tilde{r}(t)-r(t)|$, we have
\begin{equation*}
\begin{split}
|\tilde{r}(t)-r(t)|&\leq C_{12}E_{\alpha,1}(C_{12}t^\alpha)\\
&\times\left(\|\tilde{a}-a\|_{C[0,T]}+\|\tilde{g}-g\|_{\mathcal{H}^{1+\gamma}}+\|\tilde{h}-h\|_{\mathcal{H}^{2+\gamma}}+\|\tilde{E}-E\|_{X^\alpha[0,T]}\right).
\end{split}
\end{equation*}
This implies
\begin{equation}\label{cndp r}
\begin{split}
&\|\tilde{r}-r\|_{C[0,T]}\\
&\leq C_{13}\left(\|\tilde{a}-a\|_{C[0,T]}+\|\tilde{g}-g\|_{\mathcal{H}^{1+\gamma}}+\|\tilde{h}-h\|_{\mathcal{H}^{2+\gamma}}+\|\tilde{E}-E\|_{X^\alpha[0,T]}\right),
\end{split}
\end{equation}
where $C_{13}=C_{12} E_{\alpha,1}(C_{12}T^\alpha).$

Now, we are going to estimate the difference $\tilde{u}-u.$
In view of Corollary \ref{q_a>a ur source}, we have
\begin{equation}
    \begin{split}\label{cndp 0}
\|w^r(t;a;r)\|^2_{\mathcal{H}}&=\sum_{\xi\in\mathcal{I}}|w_\xi^r(t;a;r)|^2\\&=\sum_{\xi\in\mathcal{I}}\left|\int_0^t |p_\xi(\tau)|(t-\tau)^{\alpha-1}E_{\alpha,\alpha}(-\lambda_\xi q_{\tilde{a}} (t-\tau)^\alpha)d\tau\right|^2\\
&\leq \sum_{\xi\in\mathcal{I}}\left|\int_0^t |p_\xi(\tau)|(t-\tau)^{\alpha-1}E_{\alpha,\alpha}(-\inf_{\xi\in\mathcal{I}}\lambda_\xi q_{\tilde{a}} (t-\tau)^\alpha)d\tau\right|^2,
\end{split}\end{equation}
where the last inequality is obtained by using Corollary \ref{L:E1E2 ur source}. Using the Holder inequality we get
\begin{equation*}
    \begin{split}
&\sum_{\xi\in\mathcal{I}}\left|\int_0^t |p_\xi(\tau)|(t-\tau)^{\alpha-1}E_{\alpha,\alpha}(-\inf_{\xi\in\mathcal{I}}\lambda_\xi q_{\tilde{a}} (t-\tau)^\alpha)d\tau\right|^2\\
&\leq \sum_{\xi\in\mathcal{I}}\left(\int_0^t |p_\xi(\tau)|^2(t-\tau)^{\alpha-1}E_{\alpha,\alpha}(-\inf_{\xi\in\mathcal{I}}\lambda_\xi q_{\tilde{a}} (t-\tau)^\alpha)d\tau\right)\\
&\times \left(\int_0^t (t-\tau)^{\alpha-1}E_{\alpha,\alpha}(-\inf_{\xi\in\mathcal{I}}\lambda_\xi q_{\tilde{a}} (t-\tau)^\alpha)d\tau\right)\\
&\leq \left(\int_0^t \left\{\sum_{\xi\in\mathcal{I}}|p_\xi(\tau)|^2\right\}(t-\tau)^{\alpha-1}E_{\alpha,\alpha}(-\inf_{\xi\in\mathcal{I}}\lambda_\xi q_{\tilde{a}} (t-\tau)^\alpha)d\tau\right)\\
&\times \left(\int_0^t (t-\tau)^{\alpha-1}E_{\alpha,\alpha}(-\inf_{\xi\in\mathcal{I}}\lambda_\xi q_{\tilde{a}} (t-\tau)^\alpha)d\tau\right)\\
&\leq \|p\|_{C([0,T];\mathcal{H})}^2\left(\int_0^t(t-\tau)^{\alpha-1}E_{\alpha,\alpha}(-\inf_{\xi\in\mathcal{I}}\lambda_\xi q_{\tilde{a}} (t-\tau)^\alpha)d\tau\right)^2.
\end{split}\end{equation*}
Lemma \ref{L:Der E} together with \eqref{EST: Mittag} gives
\begin{equation*}\begin{split} 
&\int_0^t (t-\tau)^{\alpha-1}E_{\alpha,\alpha}(-\inf_{\xi\in\mathcal{I}}\lambda_\xi q_{\tilde{a}} (t-\tau)^\alpha)d\tau\\
&=\frac{1}{q_{\tilde{a}} \inf_{\xi\in\mathcal{I}}\lambda_\xi}(1-E_{\alpha,1}(-\inf_{\xi\in\mathcal{I}}\lambda_\xi q_{\tilde{a}} t^\alpha))<\frac{1}{q_{\tilde{a}} \inf_{\xi\in\mathcal{I}}\lambda_\xi}.
\end{split}\end{equation*}
Hence, 
\begin{equation}
    \begin{split}\label{cndp 1}
&\sum_{\xi\in\mathcal{I}}\left|\int_0^t |p_\xi(\tau)|(t-\tau)^{\alpha-1}E_{\alpha,\alpha}(-\inf_{\xi\in\mathcal{I}}\lambda_\xi q_{\tilde{a}} (t-\tau)^\alpha)d\tau\right|^2\\
&\leq \left(\frac{1}{\inf_{\xi\in\mathcal{I}}\lambda_\xi q_{\tilde{a}}}\right)^2\|p\|_{C([0,T];\mathcal{H})}^2.
\end{split}\end{equation}
Substituting \eqref{cndp 1} into \eqref{cndp 0}, we obtain
$$\|w^r(t;a;r)\|^2_{\mathcal{H}}\leq \left(\frac{1}{\inf_{\xi\in\mathcal{I}}\lambda_\xi q_{\tilde{a}}}\right)^2\|p\|_{C([0,T];\mathcal{H})}^2=C_{14}^2\|p\|_{C([0,T];\mathcal{H})}^2,$$
which gives
\begin{equation}\label{cndp 2}
    \|w^r\|_{C([0,T];\mathcal{H})}\leq C_{14} \|p\|_{C([0,T];\mathcal{H})}.
\end{equation}
Putting $p(t)=\tilde{r}(t)(\tilde{g}-g)+g(\tilde{r}(t)-r(t))-(\tilde{a}(t)-a(t))\mathcal{L}u(t;a;r)$ into \eqref{cndp 2}, we have
\begin{equation}\begin{split}\label{cndp 3}
    \|w^r\|_{C([0,T];\mathcal{H})}&\leq C_{14} \|p\|_{C([0,T];\mathcal{H})}\leq C_{14}\big(\|\tilde{r}\|_{C[0,T]}\|\tilde{g}-g\|_\mathcal{H}\\&+\|g\|_\mathcal{H}\|\|\tilde{r}-r\|_{C[0,T]}+\|\tilde{a}-a\|_{C[0,T]}\|\mathcal{L}u\|_{C([0,T];\mathcal{H})}\big)\\
    &\leq C_{15}\big(\|\tilde{g}-g\|_{\mathcal{H}^{1+\gamma}}+\|\tilde{r}-r\|_{C[0,T]}+\|\tilde{a}-a\|_{C[0,T]}\big).
  \end{split}  
\end{equation}
Substituting \eqref{cndp r} into \eqref{cndp 3}, we get
\begin{equation}\begin{split}\label{cndp 4}
\|w^r\|_{C([0,T];\mathcal{H})}&\leq C_{16}\big(\|\tilde{a}-a\|_{C[0,T]}\\&+\|\tilde{g}-g\|_{\mathcal{H}^{1+\gamma}}+\|\tilde{h}-h\|_{\mathcal{H}^{2+\gamma}}+\|\tilde{E}-E\|_{X^\alpha[0,T]}\big).
\end{split}\end{equation}
Using Corollary \ref{q_a>a ui source} and \eqref{EST: Mittag}, we have
\begin{equation*}
    \begin{split}
        \|w^i(t;a)\|_{\mathcal{H}}^2&\leq \sum_{\xi\in\mathcal{I}}|\tilde{h}_\xi-h_\xi|^2|E_{\alpha,1}(-\lambda_\xi q_a t^\alpha)|^2\\
        &\leq \sum_{\xi\in\mathcal{I}}|\tilde{h}_\xi-h_\xi|^2=\|\tilde{h}-h\|_\mathcal{H}^2,
    \end{split}
\end{equation*}
which gives
\begin{equation}\label{cndp 5}
    \|w^i\|_{C([0,T];\mathcal{H})}\leq \|\tilde{h}-h\|_\mathcal{H}.
\end{equation}
In view of $\tilde{u}(t;a;r)-u(t;a;r)=w(t;a;r)=w^i(t;a)+w^r(t;a;r)$ the estimates \eqref{cndp 4} and \eqref{cndp 5} give us
\begin{align*}
    \|\tilde{u}-u\|_{C([0,T];\mathcal{H})}&\leq C_{17}\big(\|\tilde{a}-a\|_{C[0,T]}\\&+\|\tilde{g}-g\|_{\mathcal{H}^{1+\gamma}}+\|\tilde{h}-h\|_{\mathcal{H}^{2+\gamma}}+\|\tilde{E}-E\|_{X^\alpha[0,T]}\big).
\end{align*}
This completes the proof.
\end{proof}

\subsection{Main theorem for the inverse problem}

In view of Lemma \ref{L:2}, Theorems \ref{existence source}, \ref{uniqueness source} allow us to deduce the existence and uniqueness of the solution $(r,u)$ of Problem \ref{P:ISP}. This with Theorem \ref{continuous source} gives us a well-posedness of Problem \ref{P:ISP}. From this we state the following theorem for this inverse problem.
\begin{thm}\label{Th:main theorem source}(Main theorem for the inverse problem)
   Suppose Assumption \ref{A:ISP} holds. Then the following statements hold true:\\ 
   $(a)$ There exists a unique fixed point of $K$ in $D;$\\
   $(b)$ The inverse Problem \ref{P:ISP} is well-posed.
\end{thm}

\section{Examples of operator $\mathcal{L}$}\label{S:L}

In this section, we give several examples of the settings where our direct and inverse problems are applicable. Of course, there are many other examples; here we collect those for which different types of partial differential equation are of particular importance.

\begin{itemize}
  \item {\bf Sturm-Liouville problem.}\\
  First, we describe the setting of the Sturm-Liouville operator. Let $\mathcal{L}$ be the ordinary second-order differential operator in $L^2(a,b)$ generated by the differential expression
\begin{equation}\label{SL} \mathcal{L}(u)=-u''(x),\,\,a<x<b,\end{equation} and one of the boundary conditions \begin{equation}\label{SL_B} a_1u'(b)+b_1u(b)=0,\,a_2u'(a)+b_2u(a)=0,\end{equation} or \begin{equation}\label{SL_B1} u(a)=\pm u(b),\,u'(a)=\pm u'(b),\end{equation} where $a^2_1+a^2_2>0,\,b_1^2+b_2^2>0$ and $\alpha_j,\, \beta_j,\,j=1,2,$ are some real numbers.

It is known (\cite{Naimark}) that the Sturm-Liouville problem for equation \eqref{SL} with boundary conditions \eqref{SL_B} or with boundary conditions \eqref{SL_B1} is self-adjoint in $L^2(a,b).$ It is also known that the self-adjoint problem has real eigenvalues and their eigenfunctions form a complete orthonormal basis in $L^2(a,b).$
\end{itemize}

\begin{itemize}
  \item {\bf Differential operator with involution.}\\
As a next example, we consider the differential operator with involution in $L^2(0,\pi)$ generated by the expression
\begin{equation}\label{DOI} \mathcal{L}(u)=u''(x)-\varepsilon u''(\pi-x),\,\,0<x<\pi,\end{equation} and homogeneous Dirichlet conditions \begin{equation}\label{DOI_B} u(0)=0,\,u(\pi)=0,\end{equation} where $|\varepsilon|<1$ is some real number.

The non-local functional-differential operator \eqref{DOI}-\eqref{DOI_B} is self-adjoint \cite{torebek}. For $|\varepsilon|<1,$ the operator \eqref{DOI}-\eqref{DOI_B} has the following eigenvalues
\begin{align*}\lambda_{2k}=4(1+\varepsilon)k^2,\,k\in \mathbb{N},\,\,\, \textrm{and} \,\,\, \lambda_{2k+1}=(1-\varepsilon)(2k+1)^2,\,k\in \mathbb{N}\cup\{0\},\end{align*}
and corresponding eigenfunctions
\begin{align*}&u_{2k}(x)=\sqrt{\frac{2}{\pi}}\sin{2kx},\,k\in \mathbb{N},\\& u_{2k+1}(x)=\sqrt{\frac{2}{\pi}}\sin{(2k+1)x},\,k\in \mathbb{N}\cup \{0\}.\end{align*}
\end{itemize}

\begin{itemize}
  \item {\bf Fractional Sturm-Liouville operator.}\\
We consider the operator generated by the integro-differential expression
\begin{equation}\label{FSL} \mathcal{L}(u)=\mathcal{D}_{a+}^\alpha D_{b-}^{\alpha}u,\,a<x<b,\end{equation}
and the conditions \begin{equation}\label{FSL_B}I_{b-}^{1-\alpha}u(a)=0,\,I_{b-}^{1-\alpha}u(b)=0,\end{equation} where
$\mathcal{D}_{a+}^\alpha$ is the left Caputo derivative of order $\alpha  \in
\left( {1/2,1} \right],$  $D_{b-}^\alpha$ is the right Riemann-Liouville derivative of order $\alpha \in
\left( {1/2,1} \right]$ and $I_{b-}^\alpha $ is the right Riemann-Liouville integral of order $\alpha \in
\left( {1/2,1} \right]$ (see \cite{KST06}).
The fractional Sturm-Liouville operator \eqref{FSL}-\eqref{FSL_B} is self-adjoint and positive in $L^2 (a, b)$ (see \cite{TT16}). The spectrum of the fractional Sturm-Liouville operator \eqref{FSL}-\eqref{FSL_B} is
discrete, positive, and real-valued, and the system of eigenfunctions is a complete orthogonal basis in $L^2 (a, b).$ For more properties of the operator generated by the problem \eqref{FSL}-\eqref{FSL_B} we refer to \cite{TT18, TT19}.
\end{itemize}

\begin{itemize}
    \item {\bf Second order elliptic operator $\mathcal{L}.$}\\
    Let $\Omega$ be an open bounded domain in $\mathbb{R}^d\, (d\geq 1)$ with a smooth boundary (for example, of $C^\infty$ class).

Let $L^2(\Omega)$ be the usual $L^2$-space with the inner product $(\cdot, \cdot)$ and 
let $\mathcal{L}$ be the elliptic operator defined for $g\in \mathcal{D}(-\mathcal{L}):=H^2(\Omega)\cap H^1_0(\Omega)$ as
$$
-\mathcal{L}g(x)=-\sum_{i,j=1}^{d}\partial_j(a_{ij}(x)\partial_j g(x))+c(x)g(x),\;x\in \Omega,
$$
with Dirichlet boundary condition
$$g(x)=0,\;x\in \partial \Omega,$$
where $a_{ij}=a_{ji}\,(1\leq i,j\leq d)$ and $c\geq 0$ in $\overline{\Omega}$. Moreover, assume that $a_{ij}\in C^1(\overline{\Omega}),\,c\in C(\overline{\Omega}),$ and there exists a constant $\delta>0$ such that 
$$
\delta\sum_{i=1}^d \xi_i^2\leq \delta\sum_{i=1}^d a_{ij}(x)\xi_i\xi_j, \; \forall x\in\overline{\Omega}, \,\forall (\xi_1,...,\xi_d)\in\mathbb{R}^d. 
$$
Then the elliptic operator $-\mathcal{L}$ has the eigensystem $\{\lambda_n,\omega_n\}_{n=1}^\infty$ such that $0<\lambda_1\leq\lambda_2\leq \cdots ,\lambda_n\rightarrow\infty$ as $n\rightarrow\infty$ and $\{\omega_n\}_{n=1}^\infty$ forms an orthonormal basis of $L^2(\Omega)$ (\cite{SY11}).
\end{itemize}

\begin{itemize}
  \item {\bf Harmonic oscillator.}\\
For any dimension $d\geq1$, let us consider the harmonic oscillator,
$$
\mathcal{L}:=-\Delta+|x|^{2}, \,\,\, x\in\mathbb R^{d}.
$$
The operator $\mathcal L$ is an essentially self-adjoint operator on $C_{0}^{\infty}(\mathbb R^{d})$. It has a discrete spectrum, consisting of the eigenvalues
$$
\lambda_{k}=\sum_{j=1}^{d}(2k_{j}+1), \,\,\, k=(k_{1}, \cdots, k_{d})\in\mathbb N^{d},
$$
with the corresponding eigenfunctions
$$
\varphi_{k}(x)=\prod_{j=1}^{d}P_{k_{j}}(x_{j}){\rm e}^{-\frac{|x|^{2}}{2}},
$$
which form an orthogonal basis in $L^{2}(\mathbb R^{d})$. Here $P_{l}(\cdot)$ is the $l$--th order Hermite polynomial,
$$
P_{l}(\xi)=a_{l}{\rm e}^{\frac{|\xi|^{2}}{2}}\left(x-\frac{d}{d\xi}\right)^{l}{\rm e}^{-\frac{|\xi|^{2}}{2}},
$$
where $\xi\in\mathbb R$, and
$$
a_{l}=2^{-l/2}(l!)^{-1/2}\pi^{-1/4}.
$$

\end{itemize}

\begin{itemize}
  \item {\bf Anharmonic oscillator.}\\
Another class of examples are anharmonic oscillators (see for instance
\cite{HR82}), which are operators on $L^2(\mathbb{R})$ of the form
$$
\mathcal{L}:=-\frac{d^{2k}}{dx^{2k}} +x^{2l}+p(x), \,\,\, x\in\mathbb R,
$$
for integers $k,l\geq 1$ and with $p(x)$ being a polynomial of degree $\leq 2l-1$ with real coefficients. More general case on $\mathbb{R}^n$ where a prototype operator is of the form
$$
\mathcal{L}:=-(\Delta )^k +|x|^{2l},
$$
where $k, l$ are integers $\geq 1$ see \cite{ChDR21}. 
\end{itemize}

\begin{itemize}
  \item {\bf Landau Hamiltonian in 2D.}\\
The next example is one of the simplest and most interesting models of Quantum Mechanics, that is, the Landau Hamiltonian.

The Landau Hamiltonian in 2D is given by
\begin{equation*} \label{eq:LandauHamiltonian}
\mathcal{L}:=\frac{1}{2}\left(\left(i\frac{\partial}{\partial x}-B
y\right)^{2}+\left(i\frac{\partial}{\partial y}+B x\right)^{2}\right),
\end{equation*}
acting on the Hilbert space $L^{2}(\mathbb R^{2})$, where $B>0$ is some constant. The spectrum of
$\mathcal{L}$ consists of infinite number of eigenvalues (see \cite{F28, L30}) with infinite multiplicity, of the form
\begin{equation*} \label{eq:HamiltonianEigenvalues}
\lambda_{n}=(2n+1)B, \,\,\, n=0, 1, 2, \dots \,,
\end{equation*}
and the corresponding system of eigenfunctions (see \cite{ABGM15, HH13})
{\small
\begin{equation*}
\label{eq:HamiltonianBasis} \left\{
\begin{split}
e^{1}_{k, n}(x,y)&=\sqrt{\frac{n!}{(n-k)!}}B^{\frac{k+1}{2}}\exp\Big(-\frac{B(x^{2}+y^{2})}{2}\Big)(x+iy)^{k}L_{n}^{(k)}(B(x^{2}+y^{2})), \,\,\, 0\leq k, {}\\
e^{2}_{j, n}(x,y)&=\sqrt{\frac{j!}{(j+n)!}}B^{\frac{n-1}{2}}\exp\Big(-\frac{B(x^{2}+y^{2})}{2}\Big)(x-iy)^{n}L_{j}^{(n)}(B(x^{2}+y^{2})), \,\,\, 0\leq j,
\end{split}
\right.
\end{equation*}}
where $L_{n}^{(\alpha)}$ are the Laguerre polynomials given by
$$
L^{(\alpha)}_{n}(t)=\sum_{k=0}^{n}(-1)^{k}C_{n+\alpha}^{n-k}\frac{t^{k}}{k!}, \,\,\, \alpha>-1.
$$

\end{itemize}

\begin{itemize}
  \item {\bf The restricted fractional Laplacian.}\\
  On the other hand, one can define a fractional Laplacian operator by using the integral representation in terms of hypersingular kernels,
  $$\left(-\Delta_{\mathbb{R}^n}\right)^sg(x)=C_{d,s}\, \textrm{P.V.} \int\limits_{\mathbb{R}^n}\frac{g(x)-g(\xi)}{|x-\xi|^{n+2s}}d\xi,$$ where $s\in (0,1).$

  In this case, we realize the zero Dirichlet condition by restricting the operator to act only on functions that are zero outside of the bounded domain $\Omega\subset\mathbb{R}^n.$ Caffarelli and Siro \cite{CS17} called the operator defined in such a way as the restricted fractional Laplacian $\left(-\Delta_{\Omega}\right)^s.$ Such, $\left(-\Delta_{\Omega}\right)^s$ is a self-adjoint operator in $L^2(\Omega),$ with a discrete spectrum $\lambda_{s,k}>0,\,\,k\in \mathbb{N}.$ The corresponding set of eigenfunctions $\{V_{s,k}(x)\}_{k\in \mathbb{N}},$ normalized in $L^2(\Omega),$ gives an orthonormal basis.
\end{itemize}

\section{Examples of functional $F$}\label{S:F}

In this section, as an illustration, we give several examples of the functional $F.$ Of course, there are many other examples, but here we only collect some of them.
\begin{itemize}
    \item {\bf The measurement is the total energy output from the body.}\\
    $$F[u(t,\cdot)]:=\int_\Omega u(x,t)dx,\;t\in[0,T],$$
    where $\Omega$ is a bounded subset of $\mathbb{R}^d (d\geq 1).$
\end{itemize}
\begin{itemize}
    \item {\bf Measurement at an internal point.}\\
    $$F[u(t,\cdot)]:=u(t,x^*),\;t\in[0,T],\;x^*\in \Omega,$$
    where $\Omega$ is an open bounded subset of $\mathbb{R}^d (d\geq 1).$
\end{itemize}
\begin{itemize}
    \item {\bf The measurement is the normal derivative of $u$ at one of the boundary points.}\\
    $$F[u(t,\cdot)]:=\frac{\partial u}{\partial \vec{n}}(t,x^*),\;t\in[0,T],\;x^*\in \partial\Omega,$$
    where $\Omega$ is open bounded subset of $\mathbb{R}^d (d\geq 1).$
\end{itemize}

\section{Value of $\gamma$ in particular cases of $\mathcal L$ and ${F}$}\label{S:gamma}

In this section, for the particular cases of $\mathcal L$ and ${F}$ we show how to find the value of $\gamma$ in \eqref{gamma F}.

Let $\mathcal{H}$ be $L^2(0,1)$ and let $\mathcal{L}$ be particular case of \textbf{Sturm-Liouville problem}, for example, $\mathcal{L}u=-u_{xx},\; x\in (0,1),$ with homogeneous Dirichlet boundary condition. Then the operator has the eigensystem $\{k^2, \sqrt{2}\sin{k\pi x}\}_{k\in\mathbb{N}}.$ 
\begin{itemize}
    \item Let $F$ be \textbf{the measurement is the total energy output from the body}, that is, $F[u(t,\cdot)]:=\int_{0}^1 u(t,x)dx.$ 

Since 
\begin{equation*}F[\omega_k]=\int_0^1 \sqrt{2}\sin{k\pi x}dx=\frac{1+(-1)^{k+1}}{k\pi}=\begin{cases}
    &\frac{2\sqrt{2}}{k\pi},\;\text{if},\;k=2n-1\;(n\in\mathbb{N});\\
    &0, \;\text{if},\;k=2n\;(n\in\mathbb{N}),
\end{cases}\end{equation*} 
we have $$\sum_{k\in\mathbb{N}} |F[\omega_k]|^2<\infty.$$ From this we see that $\gamma$ in \eqref{gamma F} can be taken to be $0.$
\item Let $F$ be \textbf{the measurement at an internal point}, for example, $F[u(t,\cdot)]=u(t,\frac{1}{2}).$
Then we see that
$$
F[\omega_k]=\sqrt{2}\sin{\frac{k\pi}{2}}=\begin{cases}
     &-\sqrt{2},\;\; \text{if},\;  k=4n-1,\;(n\in \mathbb{N});  \\
     & 0,\;\; \text{if},\;  k=2n,\;\;(n\in \mathbb{N});  \\
     & \sqrt{2},\;\; \text{if},\;  k=4n-3,\;\;(n\in \mathbb{N}).
\end{cases}
$$
From this we have
$$\sum_{k\in\mathbb{N}} \frac{|F[\omega_\xi]|^2}{\lambda_k}<\infty.$$ From this we see that $\gamma$ in \eqref{gamma F} can be taken to be $\frac{1}{2}.$
\item Let $F$ be \textbf{the measurement is the normal derivative of $u$ at one of the boundary points}, for example, $F[u(t,\cdot)]=u_x(t,1).$ Since
$$F[\omega_k]=\sqrt{2}k\pi \cos{k\pi}=\begin{cases}
    &-\sqrt{2}k\pi,\;\;\text{if},\;k=2n-1,\;(n\in \mathbb{N});  \\
    &\sqrt{2}k\pi,\;\;\text{if},\;k=2n,\;\;(n\in \mathbb{N}),
\end{cases}$$
we have $$\sum_{k\in\mathbb{N}} \left|\frac{F[\omega_k]}{\lambda_k}\right|^2<\infty.$$ From this we see that $\gamma$ in \eqref{gamma F} can be taken to be $1.$
\end{itemize}

\section{Appendix A}

In this section, we record several classical theorems from functional analysis used in this paper.

\begin{thm}\label{Th:BF}\cite[Theorem 1.9]{KST06}[Banach fixed point theorem] Let $X$ be a Banach space and let $A:X\rightarrow X$ be the map such that
$$\|Au-Av\|\leq\beta \|u-v\|\;\;(0 <\beta <1)$$
holds for all $u,\,v\in X.$ Then the operator $A$ has a unique fixed point $u^*\in X$ that is, $Au^*=u^*.$    
\end{thm}

\begin{thm}\label{Th:AA}\cite[Theorem 1.8]{KST06}[Arzel\'a-Ascoli theorem]
A necessary and sufficient condition that a subset of continuous functions $U,$ which are defined on the closed interval $[a,b],$ be relatively compact in $C[a,b]$ is that this subset be uniformly bounded and equicontinuous.
\end{thm}

In this statement,
\begin{itemize}
    \item $U\subset C[a,b]$ is uniformly bounded means that there exists a number $C$ such that $$|\varphi(x)|\leq C$$ for all $x\in[a,b]$ and for all $\varphi\in U,$ and
    \item $U\subset C[a,b]$ is equicontinuous means that: for every $\varepsilon>0$ there is a $\delta=\delta(\varepsilon)>0$ such that $$|\varphi(x_1)-\varphi(x_2)|<\varepsilon$$
  holds for all $x_1,\,x_2\in [a,b]$ such that $|x_1-x_2|<\delta$ and for all $\varphi\in U.$
\end{itemize}

\begin{thm}\label{Th:ShF}\cite[Theorem 1.7]{KST06}[Shauder's fixed point theorem]
    Let $U$ be a closed convex subset of $C[a,b],$ and let $A:U\rightarrow U$ be the map such that the set $\{Au:\;u\in U\}$ is relatively compact in $C[a,b].$ 
    Then the operator $A$ has at least one fixed point $u^*\in U$ i.e. $Au^*=u^*.$
\end{thm}
\begin{thm}\label{LFB}\cite[p. 77]{KF} Let $H$ be Hilbert space.
  Then linear functional $F:H\rightarrow \mathbb{R}$ is continuous if and only if it is bounded on $H.$  
\end{thm}

\end{document}